\documentclass[11pt,twoside]{amsart}
\usepackage{amssymb}

\usepackage[latin1]{inputenc}
\usepackage[T1]{fontenc}
\usepackage{mathtools}
\usepackage{graphicx}
\usepackage{tikz}
\usepackage{mathabx}
\usetikzlibrary{chains}
\usepackage[OT2,T1]{fontenc}

\usepackage{subcaption}

\usepackage{caption}
\usepackage{float}
\usepackage{mathrsfs}
\usepackage{booktabs}

\PassOptionsToPackage{pdfusetitle,pagebackref,colorlinks}{hyperref}
\usepackage{bookmark}
\hypersetup{
  linkcolor={red!70!black},
  citecolor={green!70!black},
  urlcolor={blue!80!black}
}

\usepackage[latin1]{inputenc}
\usepackage{amsmath}
\usepackage{amsthm}

\usepackage[all]{xy}
\usepackage{hyperref}
\newtheorem{theo}{Theorem}[section]

\newtheorem{prop}[theo]{Proposition}

\newtheorem{lem}[theo]{Lemma}

\numberwithin{equation}{section}

\theoremstyle{definition}
\newtheorem{defi}[theo]{Definition}
\newtheorem{rema}[theo]{Remark}
\newtheorem{examp}[theo]{Example}
\newtheorem{question}[theo]{Question}

\DeclareFontFamily{U}{mathc}{}
\DeclareFontShape{U}{mathc}{m}{it}%
{<->s*[1.03] mathc10}{}
\DeclareMathAlphabet{\mathcal}{U}{mathc}{m}{it}

\newcommand{\Aut}{{\rm Aut}}

\newcommand{\Pic}{{\rm Pic}}

\newcommand{\id}{{\rm id}}

\newcommand{\cal}{\mathcal}
\newcommand{\cA}{{\cal A}}

\newcommand{\cO}{{\cal O}}
\newcommand{\cP}{{\cal P}}

\newcommand{\bZ}{\mathbb{Z}}

\newcommand{\bR}{\mathbb{R}}
\newcommand{\bC}{\mathbb{C}}

\newcommand{\bF}{\mathbb{F}}

\newcommand{\bP}{\mathbb{P}}

\newcommand{\sE}{\mathscr{E}}

\newcommand{\sL}{\mathscr{L}}

\newcommand{\fS}{\mathfrak{S}}

\DeclareSymbolFont{cyrletters}{OT2}{wncyr}{m}{n}
\DeclareMathSymbol{\Sha}{\mathalpha}{cyrletters}{"58}

\renewcommand{\to}{\xymatrix@1@=15pt{\ar[r]&}}
\newcommand{\lto}{\xymatrix@1@=15pt{&\ar[l]}}
\renewcommand{\rightarrow}{\xymatrix@1@=15pt{\ar[r]&}}
\renewcommand{\mapsto}{\xymatrix@1@=15pt{\ar@{|->}[r]&}}
\newcommand{\mapslto}{\xymatrix@1@=15pt{&\ar@{|->}[l]&}}
\renewcommand{\twoheadrightarrow}{\xymatrix@1@=18pt{\ar@{->>}[r]&}}

\renewcommand{\hookrightarrow}{\xymatrix@1@=15pt{\ar@{^(->}[r]&}}
\newcommand{\hook}{\xymatrix@1@=15pt{\ar@{^(->}[r]&}}

\newcommand{\congpf}{\xymatrix@1@=15pt{\ar[r]^-\sim&}}
\renewcommand{\cong}{\simeq}

\newcommand{\TBC}[1]{}
\newcommand{\MD}[1]{}
\newcommand{\EV}[1]{}

\makeatletter
\def\blfootnote{\xdef\@thefnmark{}\@footnotetext}
\makeatother

\begin{document}

\title[On the inverse Galois problem for del Pezzo surfaces  of  degree 1]{On the inverse Galois problem for del Pezzo surfaces  of degree 1}

\author[L.\ Karras]{Luke Karras}

\address{Mathematisches Institut,
Universit{\"a}t Bonn, Endenicher Allee 60, 53115 Bonn, Germany}
\email{Luke.Karras0211@gmail.com}

\begin{abstract} \vspace{-2mm}  We solve the inverse Galois problem for del Pezzo surfaces of degree 1 over finite fields completely for 85 of the 112 possible types. We also determine for all 112 types the smallest field of existence. As an aside, we provide an example of a del Pezzo surface of degree 1 in characteristic 2 with more than one generalized Eckardt point. \end{abstract}

\maketitle

\tableofcontents

\section{Introduction}\label{sec: intro}

Let $X$ be a smooth projective geometrically integral surface over a field $k$. Then $X$ is called a del Pezzo surface if the anticanonical sheaf $\omega_X^{-1}$ is ample. The self-intersection $d=(\omega_X,\omega_X)$ is called the degree of $X$, where one always has $1 \leq d \leq 9$. Del Pezzo surfaces of degree 3 coincide with smooth cubic surfaces. 

Let $k$ be a perfect field, fix an algebraic closure $\overline{k}$ of $k$, and let $X$ be a del Pezzo surface of degree $d \leq 6$ over $k$. Manin \cite[Chapter IV]{manincubic} showed that the group of automorphisms of $\Pic(X \otimes \overline{k})$ preserving the intersection form and the anticanonical sheaf is isomorphic to the Weyl group of a certain root system, denoted by $E_{9-d}$. For $d \leq 3$ these are the exceptional root systems $E_8$, $E_7$ and $E_6$. The absolute Galois group $G$ of $k$ acts on $\Pic(X\otimes \overline{k})$, and also preserves the intersection form and the anticanonical sheaf. The image of $G$ in $W(E_{9-d})$ is well-defined up to conjugation. One may now ask the following question.

\begin{question}\label{question 1}
    For a given conjugacy class $C$ in $W(E_{9-d})$, does there exists a del Pezzo surface $X$ over $k$ such that the image of $G$ in $W(E_{9-d})$ belongs to $C$?
\end{question}

This question is called the \em inverse Galois problem for del Pezzo surfaces\em, and was formulated by Manin \cite[Chapter IV]{manincubic} for the special case of smooth cubic surfaces. Before we discuss the state of knowledge of this problem, let us mention two related questions. 

A theorem due to Weil \cite[Theorem 23.1]{manincubic} states that the number $N$ of $\bF_q$-rational points on a smooth projective geometrically rational surface $X$ over $\bF_q$ is given by $N=q^2+aq+1$, where $a$ is the trace of the induced action of the Frobenius automorphism on the geometric Picard group. Fixing the finite field $\bF_q$ and the degree $d$ of a del Pezzo surface, one may ask a similar inverse problem as above: Consider the possible values for the trace of an element in $W(E_{9-d})$. Which of these values are actually realized by a del Pezzo surface of degree $d$ over $\bF_q$? This version of the problem, specifically asked for cubic surfaces, is due to Serre \cite[Chapter 2.3.3]{serrelecturesNXp}.

Another kind of problem that may be asked is the so-called inverse line problem: Fix a perfect field $k$ and the degree $d$. What is the possible number of ($-1$)-curves, which are actual lines in case the degree is $\geq 3$, on a del Pezzo surface of degree $d$ over $k$ that are defined over $k$? 

We now discuss the known results of these three problems, restricting to the case where the ground field $k$ is a finite field. Note that in this case, we only have to ask for the cyclic conjugacy classes in Question \ref{question 1}.  

The first to make a contribution to Serre's question was Swinnerton-Dyer \cite{swinnerton}. Banwait, Fit\'e, and Loughran solved in \cite{Loughran19} Serre's question for all finite fields and all degrees. Additionally, they proved that for $q \gg 0$ every subgroup of $W(E_{9-d})$ is realized by some del Pezzo surface of degree $d$. Whereas the former result also heavily relies on computer verifications, the latter result is not effective, and uses quite abstract machinery. In \cite{trepalin20} and \cite{LoughranTrepalin20} Trepalin and Loughran solved the inverse Galois problem over finite fields completely in case the degree is $\geq 2$. Only a few months ago, Kaya, McKean, Streeter, and Uppal in \cite{kaya2025inversecurveproblemsdel} solved the inverse line problem for del Pezzo surfaces of arbitrary degree not only over finite fields. However, the inverse Galois problem for del Pezzo surfaces of degree 1 over finite fields is still an open problem. 

Let $X$ be a del Pezzo surface of degree 1 over a finite field $\bF_q$. The group of automorphisms of $\Pic(X \otimes \overline{\bF_q})$ preserving the intersection form and the anticanonical sheaf is isomorphic to $W(E_8)$. This group has, due to Carter \cite{Carter}, 112 cyclic conjugacy classes. We solved Question \ref{question 1} for 85 of the 112 classes. We numbered the 112 classes following Carter and say that a del Pezzo surface realizing such a class is of type $n$, where $n$ is the number of the class. We say a type exists over a finite field $\bF_q$ if a del Pezzo surface of this type exists. Now, the main theorem can be stated as follows:

\begin{theo}\label{main theo 1}
    The following holds for the types of del Pezzo surfaces of degree 1 over finite fields:
    \begin{enumerate}
        \item Types $1$ and $83$ exist over $\bF_q$ if and only if $q=16$ or $q\geq 19$.
        \item Types $2$ and $52$ exist over $\bF_q$ if and only if $q \geq 11$. 
        \item Type $8$ exists over $\bF_q$ if and only if $q \geq 8$.
        \item Types $3$, $4$, $5$, $16$, $28$, and $90$ exist over $\bF_q$ if and only if $q \geq 7$.
        \item Types $7$, $19$, and $54$ exist over $\bF_q$ if and only if $q \geq 5$.
        \item Types $6$, $9$, $10$, $14$, $15$, $29$, $31$, $34$, $41$, $64$, $85$, and $103$ exist over $\bF_q$ if and only if $q \geq 4$.
        \item Types $11$, $12$, $13$, $17$, $18$, $20$, $21$, $23$, $24$, $25$, $26$, $32$, $33$, $38$, $45$, $49$, $51$, $53$, $56$, $59$, $60$, $62$, $65$, $67$, $71$, $74$, $82$, $87$, $91$, $94$, and $102$ exist over $\bF_q$ if and only if $q \geq 3$.
        \item Types $22$, $27$, $30$, $35$, $36$, $37$, $39$, $40$, $44$, $46$, $47$, $48$, $50$, $55$, $57$, $61$, $63$, $70$, $72$, $73$, $78$, $79$, $80$, $81$, $88$, $95$, $97$, and $101$ exist over $\bF_q$ for all $q$. 
    \end{enumerate}
\end{theo}

For the remaining 27 types, we can at least determine the smallest field of existence. This in turn yields a reasonable conjecture for the whole problem, since in almost all cases a type existing over $\bF_q$ also exists over $\bF_{q'}$ for all $q' \geq q$. 

\begin{theo}\label{main theo 2}
    The following holds for the types of a del Pezzo surface of degree 1 over finite fields:
    \begin{enumerate}
        \item Types $42$, $58$, $68$, $69$, $75$, $76$, $77$, $86$, $89$, $96$, $98$, $99$, $104$, $105$, $106$, $107$, $108$, $109$, $110$, and $111$ exist over $\bF_2$.
        \item Types $43$, $66$, $93$, and $100$ exist over $\bF_3$ but not over $\bF_2$.
        \item Types $84$, $92$, and $112$ exist over $\bF_4$ but not over $\bF_2$ and $\bF_3$.
    \end{enumerate}
\end{theo}

\begin{rema}\label{rema eckardt}
    Additionally, we found an example of a del Pezzo surface of degree 1 in characteristic 2 with 15 generalized Eckardt points, cf.\ \S\! \ref{sec: alg clos}. Classically, Eckardt points are studied on cubic surfaces and are defined as points where three lines intersect, which is the maximal number of lines intersecting in one point on a smooth cubic surface. For a modern treatment, we refer to \cite{dolgachevautocubicodd} and \cite[Chapter 9]{dolgachevclassical}. On a degree 2 del Pezzo surface, the maximal number of ($-1$)-curves intersecting in one point is 4. On a degree 1 del Pezzo surface, this number depends on the characteristic of the ground field and is 16 in characteristic 2, 12 in characteristic 3 and 10 in all other characteristics. A point on a degree 1 or 2 del Pezzo surface where the maximal possible number of ($-1$)-curves intersect is called a generalized Eckardt point. According to \cite{wintergeneralized}, there are only very few examples of degree 1 del Pezzo surfaces with a generalized Eckardt point, and all these surfaces contain exactly one such generalized Eckardt point. Quite recently, Martin and Wagener found an example of a del Pezzo surface of degree 1 in characteristic 5 with 144 generalized Eckardt points, cf.\ \cite[Remark 1.6]{martinfnonsplit}. It seems that no other examples of del Pezzo surfaces of degree 1 with more than one generalized Eckardt point are known.
\end{rema}

\begin{rema}\label{rema trip}
    Very recently and independently, Trip \cite{trip} submitted a paper to the arXiv regarding the inverse Galois problem for del Pezzo surfaces over finite fields. While she focuses on surfaces with conic bundle structure and minimal del Pezzo surfaces, our aim is to solve the non-minimal cases as much as possible. Let us explain the differences and similarities in more detail:

    We also provide a proof of \cite[Theorem 1.3]{trip} using the algorithmic approach for points in general position in $\bP^2_{\bF_q}$, which was already used in \cite{Loughran19}, and rational points outside the so-called bad locus, cf.\ \cite[\S\! 5]{trip}, \S\! \ref{sec: lower bound}, and \S\! \ref{sec: blowups}.  We give a solution to the one type of index 8 not treated in \cite[Theorem 1.3]{trip}, see \S\! \ref{sec: blowups}. We do not use Magma for our results in \S\! \ref{sec: lower bound}. Instead, we only use the Carter symbol and, in case the Carter symbol does not distinguish the types, a certain cohomology group. Our results also fill in the gaps of \cite[Theorem 1.1]{trip}. 
    
    It is also important to note that our numbering is different from that one in \cite{trip}. We follow \cite{Carter}, whereas Trip follows \cite{Urabe}.   
\end{rema}

The paper is structured as follows: In \S\! \ref{sec: foundations} we recall the theory of del Pezzo surfaces that is necessary for our purposes. In \S\! \ref{sec: alg clos}, we study del Pezzo surfaces primarily over algebraically closed fields, their blow-up model, and the connection to root systems. At the end of \S\! \ref{sec: alg clos}, we briefly discuss generalized Eckardt points. In the next Section, we study the case where the ground field is non-closed and, in particular, the interplay of the action of the Galois group on the geometric Picard group and the root system. Finally, in \S\! \ref{sec: antican systems}, the anticanonical systems are under consideration. 

In \S\! \ref{sec: igp}, we prove Theorem \ref{main theo 1} and \ref{main theo 2}. We outline the proof strategy in \S\! \ref{sec: pf struc}. The proof is then outlined in \S\! \ref{sec: lower bound} to \S\! \ref{sec: sextics}. In the last section, we briefly discuss the remaining open cases. 

The appendix contains certain data associated with the groups $W(E_7)$ and $W(E_8)$ and certain normal forms for degree 1 del Pezzo surfaces used in \S\! \ref{sec: sextics}. The Magma code used for the computations in this paper, as well as the examples of del Pezzo surfaces used for the existence results can be found on the author's website \cite{Karras}. \medskip

\noindent
{\bf Acknowledgements:} Most results of this paper are based on my Master's Thesis, which was written at the University of Bonn. I heartily want to thank my advisor Gebhard Martin for selecting the topic, for his dedicated supervision, and numerous discussions. Furthermore, I am very grateful to Daniel Huybrechts for introducing me to Gebhard Martin, and for his lectures, which I was able to attend as a tutor and a student. I thank Michael Welter for providing access to the computer algebra system Magma. Last but not least, I would like to give special thanks to my great friends Christian Tack for proofreading my thesis, and David Wendler for many helpful comments and discussions regarding this work. 

\section{Foundations on del Pezzo surfaces}\label{sec: foundations}

In this section, we introduce the necessary theory for del Pezzo surfaces. For a general treatment of the topic, we refer to \cite[Chapter IV]{manincubic}, \cite{demazure}, \cite[Chapter III.3]{Kollar}, and \cite[Chapter 8]{dolgachevclassical}. Note that for the last reference, the usual assumption is that the ground field is $\bC$.

\subsection{Blow-up model, \texorpdfstring{$(-1)$}{-1}-curves, and root systems}\label{sec: alg clos}

After the definition of del Pezzo surfaces over arbitrary fields, this section focuses on del Pezzo surfaces over algebraically closed fields. We will describe such surfaces in terms of blow-ups of the projective plane and describe the connection between the Picard group of a del Pezzo surface and root systems. At the end of this section, we briefly discuss the notion of an Eckardt point and a generalized Eckardt point.

\begin{defi}\label{Definition k variety surface curve}
    Let $k$ be a field. A $k$-\emph{variety} (or variety over $k$) is a separated $k$-scheme of finite type. A \emph{surface} (resp. a \emph{curve}) over $k$ is a $k$-variety of dimension 2 (resp. 1). If $X$ is a variety over $k$, we denote by $\overline{X}$ its base change to an algebraic closure of $k$. A $k$-variety is called \emph{nice} if it is smooth projective and geometrically integral.   
\end{defi}

\begin{defi}\label{Definition del Pezzo surface}
    A \emph{del Pezzo surface} $X$ over a field $k$ is a nice surface over $k$ such that its anticanonical sheaf $\omega_X^{-1}$ is ample. The integer $(\omega_X,\omega_X)$ is called the \emph{degree} of $X$, where $(\cdot,\cdot)$ denotes the intersection form on $X$.    
\end{defi}

We next introduce a blow-up model for del Pezzo surfaces over algebraically closed fields. 

\begin{defi}\label{Definition general Position}
    Let $k$ be an algebraically closed field and $1\leq r\leq8$. We say that $r$ closed points $p_1,\dots,p_r \in \bP^2_k$ are \emph{in general position} if they satisfy the following conditions:
    \begin{enumerate}
        \item no three of them are collinear,
        \item no six of them lie on a conic,
        \item no eight of them lie on a cubic with a singularity at one of the points. 
    \end{enumerate}
\end{defi}

The following theorem is due to Manin \cite[Chapter IV, \S\! 24]{manincubic} and Demazure \cite[Surfaces de Del Pezzo - II, Th\'eor\`eme 1]{demazure}.  

\begin{theo}\label{Theorem del pezzo alg closed field points gen pos}
Let $k$ be an algebraically closed field and $X$ be a del Pezzo surface over $k$ of degree $d$. Then $1 \leq d \leq 9$ and $X$ is the blow-up of $9-d$ points in general position in $\bP^2_k$, except for $d=8$, where $X \cong \bP^1_k \times \bP^1_k$ is also possible. Conversely, any blow-up of $9-d$ points in general position in $\bP^2_k$ is a del Pezzo surface of degree $d$ over $k$.  
\end{theo}

Next, we cover $(-1)$-curves on del Pezzo surfaces and the relation between the intersection form and root systems. 

\begin{defi}\label{Definition -1 curve}
    Let $X$ be a nice surface over a field $k$. A smooth curve $C$ on $X$ (i.e., a smooth closed codimension 1 subscheme of $X$) is called a $(-1)$\emph{-curve} or \emph{exceptional curve} if $(C,C)=-1$ and $C \cong \bP^1_k$. 
\end{defi}

\begin{defi}\label{Definition N_r}
    For an integer $r \geq 1$, let $N_r = \bZ^{r+1}$ and let $L_0,\dots,L_r$ be a basis of $N_r$. Define $K_r = (-3,1,\dots,1)$ and equip $N_r$ with the bilinear form $N_r \times N_r \to \bZ$ given by the formulae
    \[(L_0,L_0)=1, \quad (L_i,L_i)=-1,\: \text{for}\: i\geq 1, \quad (L_i,L_j) =0 \: \text{for} \: i \ne j.\]
    Moreover, we define the following subsets:
    \[R_r = \{L \in N_r \mid (L,K_r)=0,(L,L)=-2\} \quad \text{and} \quad I_r=\{L\in N_r\mid(L,K_r)=(L,L)=-1\}. \]
    The elements of $I_r$ are called \emph{exceptional classes}. 
\end{defi}

The following proposition follows from standard results about monoidal transformations, cf.\ \cite[Chapter V.3]{HartshorneAG}. It can also be found in \cite[Chapter IV, \S\! 25]{manincubic}. 

\begin{prop}\label{Proposition Pic is N_r}
    Let $X$ be a del Pezzo surface of degree $d\leq 8$ over an algebraically closed field $k$ that is not isomorphic to $\bP^1_k \times\bP^1_k$. Set $r=9-d$. Then $\Pic(X)$ is free of rank $r+1$ and there exists a basis $\sL_0,\dots,\sL_r$ such that there is an isomorphism of lattices $\Pic(X) \cong N_r$ sending $\sL_i$ to $L_i$ and $\omega_X$ to $K_r$. 
\end{prop}

The next theorem relates the lattice $N_r$ to root systems. A proof can be found in \cite[Chapter IV, \S\! 25]{manincubic}. 

\begin{theo}\label{Theorem root systems and R_r}
    Let $V$ be the orthogonal complement of $K_r$ in $\bR \otimes N_r$ equipped with the induced bilinear form of $N_r$ with a negative sign. Then, for $r \leq 8$, the space $V$ is a Euclidean vector space of dimension $r$ and the vectors from the set $R_r$ form a root system of rank $r$ of $V$. Furthermore, for $3 \leq r \leq 8$, the root system $R_r$ is isomorphic to one of the following systems ordered by their ranks:
    \[A_1 \times A_2, A_4, D_5, E_6, E_7, E_8.\]
\end{theo}

The following theorem can be found in \cite[Chapter IV,\S\! 26]{manincubic}.

\begin{theo}\label{Theorem bij exc class curves, strict trafos of lines}
    Let $k$ be an algebraically closed field, $d\leq 7$ and let $f\colon X \to \bP^2_k$ be a blow-up of $r=9-d$ points $p_1,\dots,p_r$ in general position so that the resulting surface $X$ is a del Pezzo surface of degree $d$. Then there is a bijection between the set of exceptional curves and exceptional classes. The image of an exceptional curve under $f$ belongs to one of the following types:
    \begin{enumerate}
        \item one of the points $p_i$,
        \item a line passing through two of the $p_i$,
        \item a conic passing through five of the points $p_i$,
        \item a cubic passing through seven of the points $p_i$ such that one of them is a double point,
        \item a quartic passing through eight of the points $p_i$ such that three of them are double points,
        \item a quintic passing through eight of the points $p_i$ such that six of them are double points,
        \item a sextic passing through eight of the points $p_i$ such that seven of them are double points and one is a triple point. 
    \end{enumerate}
    Moreover, the number of exceptional curves depending on the degree $d$ is given by the following table: 

    \begin{table}[H]
    \centering
    \begin{tabular}{c c c c c c c c c}
    \toprule
    $1$ & $2$ & $3$ & $4$ & $5$ & $6$ & $7$ & $8$ & $9$ \\ 
    \midrule
    $240$ & $56$ & $27$ & $16$ & $10$ & $6$ & $3$ & $1$ or $0$ & $0$\\
    \bottomrule
    \end{tabular}
\end{table}
    
\end{theo}

Finally, we have the very important connection between the Weyl group of the root system $R_r$, the automorphism group of $N_r$ and the permutation group of the exceptional classes, cf.\ \cite[Chapter IV, \S\! 26]{manincubic}.

\begin{theo}\label{Theorem isomorphism weyl group and group of autom}
The following groups are isomorphic:
\begin{enumerate}
    \item The Weyl group of $R_r$,
    \item the automorphism group of $N_r$ preserving $K_r$ and the bilinear form,
    \item the permutation group of the exceptional classes preserving the bilinear form.
\end{enumerate}
\end{theo}

The next proposition is a standard fact, and follows for instance from the bijection of exceptional curves and exceptional classes in Theorem \ref{Theorem bij exc class curves, strict trafos of lines} and the bilinear form on $N_7$. 

\begin{prop}\label{Proposition 4 lines one point}
    Let $X$ be a del Pezzo surface of degree $2$. Then no more than four $(-1)$-curves intersect in one point. Moreover, for two $(-1)$-curves $L$ and $L'$, we have $(L,L')\in \{0,1,2\}$, and for each $(-1)$-curve $L$ there exists exactly one $(-1)$-curve $L'$ with $(L,L')=2$. 
\end{prop}

Similarly, one can show that on a del Pezzo surface of degree 3, or equivalently a smooth cubic surface, no more than three lines intersect in one point. 
\begin{defi}
    Let $X$ be a smooth cubic surface. A closed point on $X$ where three lines intersect is called an \emph{Eckardt point}.
\end{defi}

Surprisingly, the arguments of the above proposition fail for degree 1 del Pezzo surfaces. Instead, one has to use the ramification of the double cover $X \to Q \subseteq \bP^3$, where $Q$ is a quadratic cone, which will be briefly discussed in \S\! \ref{sec: antican systems}, and also a much more detailed analysis of the intersection graph. One gets the following result, cf.\ \cite[Theorem 1.1 and Theorem 1.2]{winterconcurrentlines}.

\begin{theo}
    Let $X$ be a del Pezzo surface of degree $1$ over an algebraically closed field $k$. If $\mathrm{char}(k) \ne 2,3$, then the maximal number of $(-1)$-curves intersecting in one point is $10$. If $\mathrm{char}(k)=3$, the maximal number of such curves is $12$, and if $\mathrm{char}(k)=2$, the maximal number of such curves is $16$. 
\end{theo}

\begin{defi}
    Let $X$ be a del Pezzo surface of degree $1$ or $2$. A closed point on $X$ where the maximal possible number of $(-1)$-curves intersect is called a \emph{generalized Eckardt point}.  
\end{defi}

According to \cite{wintergeneralized}, there are only quite a few examples of del Pezzo surfaces of degree 1 containing a generalized Eckardt point, and all these examples contain only one such point. As already mentioned in the introduction, Martin and Wagener recently found an example in characteristic 5 with 144 generalized Eckardt points.

During our work on the inverse Galois problem, we found the following example.

\begin{examp}
    Let $\bF_8=\bF_2(\beta)$ with $\beta^3+\beta+1=0$ be the field with eight elements. Let 
    \begin{align*}
    f(x,y,z,w)= w^2&+z^3+w(x^3+\beta^3x^2y+\beta^6xy^2+\beta^2 y^3)\\
    &+z(\beta x^4+\beta^4x^3y+\beta^3x^2y^2+\beta^3xy^3+y^4)+\beta^4x^5y+\beta x^3y^3.
    \end{align*}
    Then $V(f) \subseteq \bP_{\bF_{64}}(1,1,2,3)$ is a del Pezzo surface of degree 1. All 240 exceptional curves are defined over $\bF_{64}$ and it contains exactly 15 generalized Eckardt points.
\end{examp}

The statements in the example can be verified using the Magma code in \cite{Karras}. Let us briefly explain the intuition behind the example using notions that will be introduced later:

The surface above is an example of a degree 1 del Pezzo surface over $\bF_8$ of type 8. An inspection of Table \ref{Table WE8} in the appendix yields that for small finite fields the number of rational $(-1)$-curves compared to the total number of rational points on a degree 1 del Pezzo surface of type 8 is quite large. This suggests that there must be a lot of intersection points on this surface. It turns out that the example above even contains generalized Eckardt points.

\subsection{Del Pezzo surfaces over non-closed fields} \label{sec: non-closed}

In the previous section, we focused on geometric properties of del Pezzo surfaces, in particular the behavior of the $(-1)$-curves. In this section, we consider del Pezzo surfaces over non-closed fields. We will first briefly explain that being del Pezzo and the intersection form behave well under field extensions. We then discuss the Galois action on the geometric Picard group.

Let $X$ be a nice surface over a field $k$. Let $K/k$ be a field extension, and let $X_K=X \otimes K$. Recall that the intersection form on $X$ is defined as
\[(\sL,\sE) \coloneq \chi(\cO_X)-\chi(\sL^{-1})-\chi(\sE^{-1})+\chi(\sL^{-1} \otimes\sE^{-1}) \]
for line bundles $\sL$ and $\sE$ on $X$, where $\chi(\cdot)$ is the Euler characteristic. By flat base change, we obtain $(\sL,\sE)=(\sL\otimes K, \sE \otimes K)$. Using the cohomological criterion for ampleness, it is easy to see that a line bundle $\sL$ on $X$ is ample if and only if $\sL \otimes K$ is ample. Finally, we have $\omega_{X_K} \cong \omega_X \otimes K$. Since being smooth and projective is stable under base change, we obtain the following two result. 

\begin{prop}\label{Proposition del pezzo well behaved base change}
    Let $X$ be a nice surface over a field $k$. Then the following are equivalent:
    \begin{enumerate}
        \item $X$ is a del Pezzo surface of degree $d$,
        \item $X_K$ is a del Pezzo surface of degree $d$ for some field extension $K/k$,
        \item $X_K$ is a del Pezzo surface of degree $d$ for all field extensions $K/k$.
    \end{enumerate}
\end{prop}

\begin{prop}\label{Proposition Gal preserves int form}
    Let $K/k$ be a Galois extension and $X$ a nice surface over $k$. Then ${\rm Gal}(K/k)$ acts on $\Pic(X_K)$ via pullback, and this action preserves the intersection form and the canonical sheaf.
\end{prop}

Let $X$ be a del Pezzo surface over a perfect field $k$, and let $G$ be the absolute Galois group. Then, by the above proposition, we get a homomorphism of groups $G \to \mathrm{Aut}(\mathrm{Pic}(\overline{X}))$. After choosing a basis as in Proposition \ref{Proposition Pic is N_r}, we obtain a homomorphism $G \to W(E_N)$ whose image is well defined up to conjugation. If $k$ is a finite field, the image is a cyclic subgroup. We can now precisely formulate the inverse Galois problem for finite fields.
\begin{question}[Inverse Galois problem for del Pezzo surfaces over finite fields]
    Let $1 \leq d\leq 9$ be an integer, $k$ a finite field with absolute Galois group $G$. For a given cyclic conjugacy class in $W(E_{9-d})$: Does there exist a del Pezzo surface of degree $d$ over $k$ such that the image of $G$ in $W(E_{9-d})$ belongs to this conjugacy class?
\end{question}

The classification of these cyclic conjugacy classes is due to Carter \cite{Carter}. For $W(E_7)$ and $W(E_8)$, we have summarized the necessary data in Tables \ref{Table WE7} and \ref{Table WE8} in the appendix. In the context of del Pezzo surfaces, we use the following definition.  

\begin{defi}\label{Definition Type of del pezzo}
Let $X$ be a del Pezzo surface of degree $d$ over a finite field. The \emph{type} of $X$ is the number of the cyclic conjugacy class as listed in Tables \ref{Table WE7} and \ref{Table WE8} in the appendix associated with the Galois action on $\Pic(\overline{X})$. We say that a certain type exists if there exists a del Pezzo surface of this type.  
\end{defi}

To each conjugacy class Carter associates a certain symbol, the so-called Carter symbol (in \cite{Carter} this is called admissible diagram). Each symbol represents a certain graph associated to a cyclic conjugacy class. The precise meaning of the symbol is not important to us except in the next lemma, see \cite[Lemma 2.10]{trepalin20}. For the sake of completeness, we give a proof of this lemma.

\begin{lem}\label{Lemma carter symbol blow up}
    Let $X$ be a del Pezzo surface over a finite field with Carter symbol $R$. Suppose the blow-up $Y$ of $X$ in a point of degree $n$ is again a del Pezzo surface. Then the Carter symbol of $Y$ is $R\times A_{n-1}$. 
\end{lem}

\begin{proof}
    Let $E_1, \dots,E_n$ be the exceptional curves associated to the blow-up of $\overline{X}$. Hence
    \[\mathrm{Pic}(\overline{Y}) = \mathrm{Pic}(\overline{X}) \oplus \bZ E_1 \oplus \dots \oplus \bZ E_n. \]
    The action of the absolute Galois group on $\mathrm{Pic}(\overline{Y})$ is thus determined by the action on $\mathrm{Pic}(\overline{X})$ and on $\bZ E_1 \oplus \dots \oplus \bZ E_n$. On $\mathrm{Pic}(\overline{X})$ the Carter symbol associated to the action is $R$. On $\bZ E_1 \oplus \dots \oplus \bZ E_n$ the action is cyclic and according to Carter the associated graph is $A_{n-1}$. Thus, the Carter symbol has to be $R \times A_{n-1}$.  
\end{proof}

The following lemma follows easily from the definition of the strict transform. 

\begin{lem}\label{Lemma strict trafo under galois}
    Let $X$ be a del Pezzo surface over $\bF_q$, and suppose $X$ is obtained by blowing up closed points in $\bP^2_{\bF_q}$. Let $C$ be one of the plane curves described in Theorem \ref{Theorem bij exc class curves, strict trafos of lines}. Let $\sigma \in {\rm Gal}(\overline{\bF_q}/\bF_q)$. Then for the strict transforms, we have $\sigma(\widetilde{C}) = \widetilde{\sigma(C)}$.
\end{lem}

The next theorem can be found in \cite[Theorem 9.3.3]{poonen}.

\begin{theo}\label{Theorem contraction of an galois orbit}
    Let $X$ be a nice surface over a field $k$. Let $\{C_1,\dots,C_n\}$ be a Galois-invariant set of pairwise disjoint $(-1)$-curves in $\overline{X}$. Then these can be contracted, i.e., there exists a nice surface $X'$ over $k$ and a birational morphism $f \colon X \to X'$ such that $f$ is the composition of $n$ blow-ups, where each of the curves $C_i$ is the exceptional divisor of one blow-up. 
\end{theo}

We finish this section with a theorem due to Weil \cite[Theorem 23.1]{manincubic}.

\begin{theo}\label{Theorem Weil number of points}
    Let $X$ be a smooth projective geometrically rational surface over a finite field $\bF_q$. Let $\sigma^*\colon \Pic(\overline{X}) \to \Pic(\overline{X})$ be the induced action of the Frobenius automorphism on the Picard group. Then, for the number $N$ of $\bF_q$-rational points on $X$, we have 
    \[N = q^2+\mathrm{Tr}(\sigma^*)q+1.\]
\end{theo}

\subsection{Anticanonical Systems}\label{sec: antican systems}

Let $X$ be a del Pezzo surface with anticanonical divisor $-K_X$. A detailed study of the linear systems $|-nK_X|$ yields two interesting applications for del Pezzo surfaces of degree 1 and 2: Firstly, they enable us to deduce concrete equations for del Pezzo surfaces. Secondly, we can realize in this way a degree 2 (resp. 1) del Pezzo surface as a double cover of $\bP^2$ (resp. a quadratic cone in $\bP^3$), which can be used to construct a twist of the surface. 

For the following theorem, we refer to \cite[Chapter III.3]{Kollar}.

\begin{theo}
    Let $X$ be a del Pezzo surface over a field $k$. Then the following holds:
    \begin{enumerate}
        \item If $K_X^2=2$, then $X$ is isomorphic to a smooth hypersurface of degree $4$ in $\bP_k(1,1,1,2)$. Conversely, every such hypersurface is a del Pezzo surface of degree $2$. 
        \item If $K_X^2=1$, then $X$ is isomorphic to a smooth hypersurface of degree $6$ in $\bP_k(1,1,2,3)$. Conversely, every such hypersurface is a del Pezzo surface of degree $1$.  
    \end{enumerate}
\end{theo}

The first part of the next proposition can be found in \cite[Surfaces de Del Pezzo - V.]{demazure}. For the second part, we refer to \cite[Proposition 3.1 and Proposition 5.1]{gebhardchar2}.

\begin{prop}\label{Proposition double cover}
    Let $X$ be a del Pezzo surface of degree $2$ (resp. degree $1$). Then $|-K_X|$ (resp. $|-2K_X|$) induces a finite morphism $X \to \bP^2$ (resp. $X \to Q \subseteq \bP^3$, where $Q$ is a quadratic cone in $\bP^3$) of degree $2$. The morphism is separable in all characteristics.
\end{prop}

The double covers from the above proposition induce an involutional automorphism of a degree 2 (resp. degree 1) del Pezzo surface $X$, called the {\bf Geiser} (resp. {\bf Bertini}) {\bf involution}. We have that $\Aut(X)$ injects into $W(E_7)$ (resp. $W(E_8)$), cf.\ \cite[Chapter 8.2.8]{dolgachevclassical}. The proof does not depend on the characteristic of the ground field. It can be shown that the image of the Geiser (resp. Bertini) involution in $W(E_7)$ (resp. $W(E_8)$) corresponds to the unique involution in the center of $W(E_7)$ (resp. $W(E_8)$), cf.\ \cite[Chapter 8.7.2 and 8.8.2]{dolgachevclassical}, which again does not depend on the characteristic of the ground field. This involution has the following properties, see again \cite{dolgachevclassical}. 

\begin{prop}\label{Proposition geiser bertini intersection -id}
 Let $w_0\in W(E_7)$ (resp. $W(E_8)$) be the unique central involution. Then $w_0=-\id_{E_7}$ (resp. $-\id_{E_8}$). Let $X$ be a degree $2$ del Pezzo surface, $L$ a $(-1)$-curve, and denote the Geiser involution by $\gamma$. A straightforward calculation shows that $(L,\gamma(L))=2$. 
\end{prop}

For the ramification locus of the double cover $X \to \bP^2$, we have the following lemma, cf.\ \cite[Lemma 4.1]{Loughran19}.

\begin{lem}\label{Lemma ramification curve}
    Let $X$ be a del Pezzo surface of degree $2$ over an algebraically closed field $k$ with ramification locus $R$. The following holds:
    \begin{enumerate}
        \item If $\mathrm{char}(k) \ne 2$, then $R$ is a smooth curve of genus $3$. 
        \item If $\mathrm{char}(k)=2$, then $R$ has genus $0$ and at most two irreducible components. 
    \end{enumerate}
\end{lem}

The Geiser (resp. Bertini) involution can be used to twist a del Pezzo surface $X$ of degree 2 (resp. 1) over $\bF_q$. Recall that a twist of a variety $X$ over a field $k$ is a variety $X'$ over $k$ such that $X \cong X'$ over $\overline{k}$.

Let $X$ be a del Pezzo surface of degree 2 (resp. 1) over $\bF_q$. Then the group $\bZ/2\bZ$ acts faithfully on $X$ and $\mathrm{Pic}(\overline{X})$ via the Geiser (resp. Bertini) involution. By \cite[Proposition 4.4]{trepalinmincubic}, this action induces a nontrivial twist $X'$ of $X$. This twist is called the \emph{Geiser} (resp. \emph{Bertini}) \emph{twist}. 

Furthermore, let $w$ and $w'$ be the elements in the Weyl group associated with the types of $X$ and $X'$. Then $w$ as well as $w'$ can be considered as linear maps of the vector space containing $E_7$ (resp. $E_8$). It is a consequence of \cite[Proposition 4.4]{trepalinmincubic} that the eigenvalues of the characteristic polynomial of $w'$ can be obtained by multiplying the eigenvalues of the characteristic polynomial of $w$ by $-1$.

In particular, for a solution of the inverse Galois problem, we only have to classify types up to the Bertini twist. We have computed the Bertini twists for each type, and the data can be found in Table \ref{Table WE8} in the appendix. Observe that the behavior of the Geiser and Bertini twist is slightly different. Whereas in the first case, the Geiser twist divides the 60 conjugacy classes into 30 pairs, cf.\ \cite{trepalin20}, this is not true for the Bertini twist. There are several types where the Bertini twist has the same type again. Also, a few types cannot be distinguished just by the eigenvalues. It turns out that in these cases, the Bertini twist of such a type is again the type itself. This will be a consequence of Theorem \ref{main theo 1}. 

We end this section with a result about blow-ups of del Pezzo surfaces. 

\begin{prop}\label{Proposition when is blow up dP}
    Let $X$ be a del Pezzo surface over a field $k$ of degree $d\geq 3$. The blow-up of a $k$-rational point $p$ is a del Pezzo surface of degree $d-1$ if and only if $p$ does not lie on the $(-1)$-curves. If $X$ is a del Pezzo surface of degree $2$ then the blow-up of a $k$-rational point is a del Pezzo surface of degree $1$ if and only if $p$ lies neither on the $(-1)$-curves nor on the ramification curve. 
\end{prop}

\begin{proof}
    The first part of the proposition is an obvious corollary of Theorem \ref{Theorem bij exc class curves, strict trafos of lines}. For the second part, we refer to \cite[Corollary 14]{unirationaldelpezzodegree2}.
\end{proof}

\section{The inverse Galois problem}\label{sec: igp}

In this section, we prove Theorem \ref{main theo 1} and Theorem \ref{main theo 2}. In \S\! \ref{sec: pf struc}, we first discuss the structure of the proof. The proof is then outlined in \S\! \ref{sec: lower bound} to \S\! \ref{sec: sextics}. Lastly, we briefly consider the remaining open cases.  

\subsection{Proof structure}\label{sec: pf struc}

As already remarked in the introduction, we solve the inverse Galois problem for del Pezzo surfaces over finite fields for 85 of the 112 cases. More precisely, this solves the problem for all types that are either blow-ups of $\bP^2_{\bF_q}$ or contain a $\bF_q$-rational $(-1)$-curve, or are the Bertini twist of one of the latter two.

For those types that contain a rational $(-1)$-curve, we derive a lower bound for the existence over finite fields by using the already known solution to the inverse Galois problem for del Pezzo surfaces of degree 2. This will be done in \S\! \ref{sec: lower bound}. We are thus left to study the problem over relatively small finite fields. We present two methods for this:
\begin{enumerate}
    \item Blow-ups of $\bP^2_{\bF_q}$
    \item Sextics in $\bP_{\bF_q}(1,1,2,3)$
\end{enumerate}

In \S\! \ref{sec: blowups}, we consider all those types that can be realized as a blow-up of $\bP^2_{\bF_q}$. We recall in detail an algorithmic approach to find eight points in general position or prove their non-existence in $\bP^2_{\bF_q}$  that was already used in \cite{Loughran19}.

In \S\! \ref{sec: sextics}, we use the equations for del Pezzo surfaces of degree 1. The first step is to associate to each type data, which defines the type uniquely, and can be efficiently computed for a given equation. This data will be the number of $(-1)$-curves and rational points over the ground field and certain field extensions. The second step is to derive normal forms for the equations over small finite fields such that over each finite field, we have to consider as few equations as possible. A loop over these forms either yields an example of a del Pezzo surface of degree 1 of the desired type or shows the non-existence of that type over the specific finite field.  

Combining the lower bounds from Table \ref{Table lower bound}, the examples of degree 1 del Pezzo surfaces listed in \cite{Karras}, and the Magma code looping over all normal forms proves Theorem \ref{main theo 1} and \ref{main theo 2}. 

\begin{rema}
    In our thesis, we elaborated a third method to study certain types over small finite fields. The idea is to generalize the algorithm for points in general position in $\bP^2_{\bF_q}$ to non-split quadric surfaces. This indeed works out but, unfortunately, can only be used to solve the types 17, 33, 36, and 61 (and their Bertini twists). Since these types can also be solved using equations, we will not discuss this approach here in more detail.  
\end{rema}

\subsection{A lower bound}\label{sec: lower bound}

By Proposition \ref{Proposition when is blow up dP}, the blow-up of a del Pezzo surface of degree 2 in a $k$-rational point is a del Pezzo surface of degree 1 if and only if the point lies outside the ramification curve and the $(-1)$-curves. The goal of this section is to bound the number of $k$-rational points on the $(-1)$-curves and on the ramification curve.

\begin{prop}
    Let $X$ be a del Pezzo surface of degree $2$ over a finite field $\bF_q$. Let $n$ be the number of $\bF_q$-rational $(-1)$-curves and $m$ the number of Galois orbits of $(-1)$-curves of length $2$, $3$ and $4$. Then, for the number $B(X)$ of $\bF_q$-rational points on the $(-1)$-curves and on the ramification curve, we have
   \[B(X) \leq \begin{cases}
       n(q+1)+m+q+1+6\sqrt{q}, \: \text{if} \: q \: \text{is odd},\\
       n(q+1)+m+2(q+1), \: \text{if} \: q \: \text{is even}. 
   \end{cases}\]
\end{prop}

\begin{proof}
    The number of rational points on the ramification curve is bounded by $q+1+6\sqrt{q}$ and $2(q+1)$ by Lemma \ref{Lemma ramification curve} and the Hasse-Weil bound. Proposition \ref{Proposition 4 lines one point} yields that $(-1)$-curves belonging to an orbit of length greater than 4 cannot contain $\bF_q$-points. The same proposition also yields that an orbit of length 3 or 4 contains at most one $\bF_q$-point. An orbit of length 2 may contain one or two $\bF_q$-points. In the latter case, according to Proposition \ref{Proposition geiser bertini intersection -id}, these points must lie on the ramification curve and are thus encountered already in that part of the bound. The remaining part of the bound follows from the fact that a $\bF_q$-rational $(-1)$-curve is isomorphic to $\bP^1_{\bF_q}$ and hence has $q+1$ rational points. 
\end{proof}

\begin{rema}
    The basic idea of this bound is already outlined in \cite[\S\! 5.2]{Loughran19}. They call the points on the $(-1)$-curves and on the ramification curve the \emph{bad locus}. 
\end{rema}

The following theorem summarizes the results of the inverse Galois problem for degree 2 del Pezzo surfaces over finite fields, cf.\ \cite{trepalin20} and \cite{LoughranTrepalin20}. Our numbering of the types is compatible with that in these articles. 

\begin{theo}
    The following holds for the types of del Pezzo surfaces of degree $2$:
    \begin{itemize}
        \item Types $1$ and $49$ exist over $\bF_q$ if and only if $q \geq 9$.
        \item Types $5$ and $10$ exist over $\bF_q$ if and only if $q \geq 7$.
        \item Types $2$, $3$, $18$ and $31$ exist over $\bF_q$ if and only if $q \geq 5$. 
        \item Types $4$, $6$-$9$, $12$-$14$, $17$, $21$, $22$, $25$, $28$, $32$, $33$, $35$, $38$, $40$, $50$, $53$, $55$ and $60$ exist over $\bF_q$ if and only if $q \geq 3$.
    \end{itemize}
    All other types exist over $\bF_q$ for all $q$. 
\end{theo}

Thus, for each type of a degree 2 del Pezzo surface, we can compute values $q_0$ (resp. $q_0'$) for odd (resp. even) prime powers such that for all $q \geq q_0,q_0'$ a degree 2 del Pezzo surface over $\bF_q$ of this type admits a $\bF_q$-rational point outside the $(-1)$-curves and the ramification curve using above proposition and Theorem \ref{Theorem Weil number of points}. These values are summarized at the end of this section in Table \ref{Table lower bound}, where the type refers to degree 1 del Pezzo surfaces. Hence, the corresponding type of a degree 1 del Pezzo exists for $q \geq q_0,q_0'$ if and only if the type of the degree 2 del Pezzo surface exists. The only if part follows from Castelnuovo's contraction criterion (see Theorem \ref{Theorem contraction of an galois orbit}). Observe that by Theorem \ref{Theorem contraction of an galois orbit}, the existence of a degree 2 type is necessary for the existence of the corresponding degree 1 type for all $q$. This already proves certain non-existence claims in Theorem \ref{main theo 1}.

Usually, the type of a degree 1 del Pezzo surface obtained by blowing up a certain type of degree 2 del Pezzo surface can be uniquely determined by the Carter symbol (see Lemma \ref{Lemma carter symbol blow up}). There are a few cases where the type either on a degree 2 or on a degree 1 del Pezzo surface is not uniquely determined by the Carter symbol. Table \ref{Table carter all} below lists the relevant types. 

\begin{table}[H]
    \centering
    \caption{Types not uniquely determined by the Carter symbol}
    \label{Table carter all}
    \vspace{2mm}
    
    \begin{subtable}[t]{0.45\textwidth}
        \centering
        \caption{Degree 2}
        \label{Table Carter  deg 2}
        \begin{tabular}[t]{c  c}
        \toprule
        Type & Carter Symbol\\
        \midrule
        5 & $(A_1^3)'$\\
        6 & $(A_1^3)''$\\
        9 & $(A_1^4)'$ \\
        10 & $(A_1^4)''$ \\
        13 & $(A_3 \times A_1)'$ \\
        14 & $(A_3\times A_1)''$ \\
        21 & $(A_3\times A_1^2)'$ \\
        22 & $(A_3\times A_1^2)''$ \\
        25 & $(A_5)'$ \\
        26 & $(A_5)''$ \\
        37 & $(A_5 \times A_1)'$ \\
        38 & $(A_5\times A_1)''$ \\
        \bottomrule
        \end{tabular}
    \end{subtable}
    \hfill 
    \begin{subtable}[t]{0.45\textwidth}
        \centering
        \caption{Degree 1}
        \label{Table Carter deg 1}
        \begin{tabular}[t]{c  c}
        \toprule
        Type & Carter Symbol\\
        \midrule
        8 & $(A_1^4)'$ \\
        9 & $(A_1^4)''$ \\
        19 & $(A_3\times A_1^2)'$ \\
        20 & $(A_3\times A_1^2)''$ \\
        34 & $(A_3^2)'$ \\
        35 & $(A_3^2)''$ \\
        38 & $(A_5 \times A_1)'$ \\
        39 & $(A_5 \times A_1)''$ \\
        62 & $(A_7)'$ \\
        63 & $(A_7)''$ \\
        \bottomrule
        \end{tabular}
    \end{subtable}
    
\end{table}

For a degree 1 del Pezzo surface of type 34, 35 or 62, 63 there is no confusion since only one of the two with the same Carter symbol contain a rational $(-1)$-curve. 

For the other degree 1 types, we use the following fact: Let $X$ be a smooth projective surface over a field $k$ with absolute Galois group $G$. Since $\mathrm{Pic}(\overline{X})$ is a $G$-module, we can consider the cohomology group $H^1(G,\mathrm{Pic}(\overline{X}))$. It turns out that this group is a birational invariant, see \cite[Theorem 29.1]{manincubic}. The values of this group of the types above are summarized in the appendix in Table \ref{Table H1} and can be used to distinguish the types 8, 9, 19, 20, 38 and 39. 

The remaining types are two different degree 2 types that blow up to the same degree 1 type. The degree 2 type 5 exists for $q \geq7$ and also admits a point outside the $(-1)$-curves and the ramification curve for $q \geq 7$. Hence, it is enough to consider the degree 1 type 5 for $q \leq 5$. Similarly, it is enough for degree 1 type 12 to consider it for $q \leq 5$. Finally, degree 1 type 23 must only be checked over $\bF_2$. 

\begin{table}[H]
    \caption{Lower Bounds}
    \label{Table lower bound}
    \centering
    \begin{tabular}{@{} c *{15}{c} @{}}
    \toprule
    Type & 1 & 2 & 3 & 4 & 5 & 6 & 7 & 8 & 9 & 10 & 11 & 12 & 13 & 14 & 15 \\
    \midrule
    $q_0$  & 53 & 31 & 17 & 23 & 7 & 11 & 13 & 17 & 9 & 9 & 7 & 7 & 9 & 11 & 11 \\ 
    $q_0'$ & 64 & 32 & 32 & 32 & 8 & 16 & 16 & 16 & 8 & 8 & 8 & 8 & 8 & 16 & 16 \\
    
    \addlinespace 
    \midrule
    \addlinespace
    
    Type & 16 & 17 & 18 & 19 & 20 & 21 & 22 & 23 & 24 & 25 & 26 & 27 & 28 & 31 & 32 \\
    \midrule
    $q_0$  & 9 & 7 & 7 & 11 & 7 & 5 & 7 & 3 & 5 & 5 & 7 & 7 & 11 & 9 & 9 \\ 
    $q_0'$ & 8 & 8 & 8 & 16 & 8 & 4 & 8 & 4 & 4 & 8 & 8 & 8 & 16 & 8 & 8 \\
    
    \addlinespace
    \midrule
    \addlinespace
    
    Type & 33 & 34 & 37 & 38 & 39 & 40 & 41 & 44 & 45 & 46 & 47 & 48 & 49 & 50 & 51 \\
    \midrule
    $q_0$  & 7 & 7 & 5 & 7 & 5 & 5 & 7 & 5 & 5 & 5 & 5 & 3 & 7 & 5 & 5 \\ 
    $q_0'$ & 8 & 8 & 4 & 8 & 4 & 2 & 8 & 4 & 4 & 4 & 2 & 2 & 8 & 4 & 4 \\
    
    \addlinespace
    \midrule
    \addlinespace
    
    Type & 52 & 56 & 60 & 62 & 64 & 70 & 71 & 78 & 79 & 80 & 81 & 82 & & & \\
    \midrule
    $q_0$  & 13 & 9 & 5 & 5 & 9 & 7 & 5 & 5 & 5 & 5 & 3 & 3 & & & \\ 
    $q_0'$ & 16 & 8 & 4 & 4 & 8 & 8 & 4 & 4 & 4 & 4 & 2 & 2 & & & \\
    \bottomrule
    \end{tabular}
\end{table}

From Table \ref{Table lower bound}, we see in particular that the types 48, 81 and 82 do not have to be studied further since their existence is completely determined by the corresponding degree 2 type.

\subsection{Blow-ups of \texorpdfstring{$\bP^2_{\bF_q}$}{P\string^2\_F\_q}}\label{sec: blowups}

In this section, we consider all those types that are blow-ups of $\bP^2_{\bF_q}$. Using Lemma \ref{Lemma carter symbol blow up}, we can identify them using the Carter symbol. As we have seen in the last section, there is one small subtlety: For a few types, we have the same Carter symbol even though the types are different. Using Lemma \ref{Lemma strict trafo under galois} and Theorem \ref{Theorem bij exc class curves, strict trafos of lines}, one can explicitly determine the orbits of $(-1)$-curves and hence decide which type is a blow-up of $\bP^2_{\bF_q}$. The types 1, 35 and 63 are treated at the end of this section. 

The following algorithm can be found in the Magma code of \cite{Loughran19}. We will describe the algorithm in detail. The idea is to translate the conditions for a set of points to be in general position into linear algebra criteria that can easily be verified by a computer search. To obtain different types from blowing up the plane, we have to blow-up different Galois orbits. Hence, one just lists all possible orbits and tests the linear algebra criteria to either find an example or to prove non-existence.

The next lemma is a straightforward check, cf.\ \cite[Lemma 2.5]{Loughran19}.

\begin{lem}\label{Lemma matrix cond gen pos}
    Let $k$ be a field, and let $P_i = [x_i:y_i:z_i] \in \bP^2_k$ be $r$ distinct $k$-rational points for $1 \leq r \leq 8$. 
    \begin{itemize}
        \item If $r=3$, then $P_1$, $P_2$ and $P_3$ are collinear if and only if the matrix 
        \[\begin{pmatrix}
            x_1 & y_1 & z_1 \\
            x_2 & y_2 & z_2\\
            x_3 & y_3 & z_3
        \end{pmatrix}\]
        has determinant 0.
        \item If $r=6$, then $P_1,\dots,P_6$ lie on a conic if and only if the matrix $M \in M_{(6,6)}(k)$ whose $i$th row is 
        \[\begin{pmatrix}
            x_i^2 & y_i^2 & z_i^2 & x_iy_i & x_iz_i & y_iz_i
        \end{pmatrix}\]
        has determinant 0.
        \item If $r=8$, for each $1 \leq i \leq 8$, consider the matrix $M_i \in M_{(11,10)}(k)$ whose $j$th row for $1 \leq j \leq 8$ is 
        \[\begin{pmatrix}
            x_j^3 & y_j^3 & z_j^3 & x_j^2y_j & x_j^2z_j & x_jy_j^2 & y_j^2z_j & x_jz_j^2 & y_jz_j^2 & x_jy_jz_j
        \end{pmatrix}\]
        and whose last three rows are 
        \[\begin{pmatrix}
            3x_i^2 & 0 & 0 & 2x_iy_i & 2x_iz_i & y_i^2 & 0 & z_i^2 & 0 & y_iz_i\\
            0 & 3y_i^2 & 0 & x_i^2 & 0 & 2x_iy_i & 2y_iz_i & 0 & z_i^2 & x_iz_i\\
            0 & 0 & 3z_i^2 & 0 & x_i^2 & 0 & y_i^2 & 2x_iz_i & 2y_iz_i & x_iy_i
        \end{pmatrix}.\]
        Then $P_1,\dots,P_8$ lie on a cubic with a singularity at $P_i$ if and only if $M_i$ has a non-trivial kernel. 
     \end{itemize}
\end{lem}

The following two lemmas are used in Algorithm 1.

\begin{lem}\label{Lemma four rational points P2}
    Let $k$ be a field and let $P_1,\dots,P_4 \in \bP^2_k$ be four $k$-rational points in general position. Then there exists a projective linear automorphism $\varphi \in \mathrm{PGL}_3(k)$ such that $\varphi(P_1)=[1:0:0]$, $\varphi(P_2)=[0:1:0]$, $\varphi(P_3)=[0:0:1]$ and $\varphi(P_4)=[1:1:1]$.
\end{lem}

\begin{proof}
    Let $V$ be a three-dimensional $k$-vector space and let $\langle v_1\rangle,\dots,\langle v_4\rangle$ be the corresponding linear subspaces to $P_1,\dots,P_4$. By assumption $v_1$, $v_2$ and $v_3$ are linearly independent and hence there exist $a,b,c \in k$ such that $v_4 = av_1+bv_2+cv_3$. Let $T\colon V \to V$ be the linear map sending $v_i$ to $e_i$ for $i \in \{1,2,3\}$. Then $T(v_4)=(a,b,c)$. Composing with $\mathrm{diag}(a^{-1},b^{-1},c^{-1})$ yields the claim. 
\end{proof}

\begin{lem}\label{Lemma enough lines in P2}
    Let $k=\bF_q$ be a finite field and let $S$ be a set of at most eight $k$-rational points in $\bP_k^2$ such that no three of them lie on a line. Then there exists a projective line $L \subseteq \bP^2_k$ disjoint from $S$. In particular, the points in $S$ can be mapped projectively linearly into the affine chart $\{z\ne0\}$. 
\end{lem}

\begin{proof}
    It is enough to consider the case where $|S|=8$. Each $k$-rational point is contained in exactly $q+1$ $k$-rational lines. Denote by $l_1$ and $l_2$ the number of lines through exactly one point of $S$ or two points of $S$, respectively. Since no three points lie on a line, we obtain $2l_2+l_1=8(q+1)$. Since $l_2 = \binom{8}{2} = 28$, we get $l_1 = 8(q+1)-56$. There are $q^2+q+1$ $k$-rational lines in $\bP^2_k$. If each such line would meet a point in $S$, then 
    \[q^2+q+1=l_1+l_2 = 28+8(q+1)-56=8q-20.\]
    But this equation has no solution and consequently there exists a line disjoint from $S$. 
\end{proof}

\noindent\textbf{Algorithm 1:} Points in general position in $\bP^2_{\bF_q}$ \label{alg: 1}

\noindent\textbf{Input:} A prime power $q$ and a sequence $n_1,\dots,n_r$ of non-negative integers 

\noindent\textbf{Output:} Either a set $M$ of $\sum_j n_j$ points in general position in $\bP^2_{\bF_q}$ such that $n_j$ points are of degree $j$ or \textbf{False} if such a set of points does not exist

\begin{enumerate}
    \item Depending on the value of $n_1$ we can assume without loss of generality that $M$ contains the points $[1:0:0]$, $[0:1:0]$, $[0:0:1]$ and $[1:1:1]$
    \item Construct the set $\cA = \underbrace{\bF_q^\times\times \dots\times \bF_q^\times}_{n_1-4 \: \text{times}}\times \dots \times \underbrace{\bF_{q^r}^\times\times \dots \times \bF_{q^r}^\times}_{n_r \:\text{times}}$ (here $n_1-4$ is considered to be 0 if it is $< 0$)
    \item For elements $x=(x_i),y=(y_i) \in \cA$ construct for all $i$ and $0 \leq t \leq r-1$ the points $[x_i^{q^t}:y_i^{q^t}:1]$
    \item Apply Lemma \ref{Lemma matrix cond gen pos} to the set $M$ of points from step 3 if $|M| = \sum_jn_j\cdot j$
    \item Loop over all possible combinations of elements $x,y\in \cA$ and do steps 3 and 4
    \item If a set $M$ of points in general position is found return $M$ and if no such combination exists return \textbf{False}
\end{enumerate}

\begin{prop}
    Algorithm $1$ is correct.
\end{prop}

\begin{proof}
    The correctness of step 1 follows from Lemma \ref{Lemma four rational points P2}. Furthermore, by Lemma \ref{Lemma enough lines in P2}, we can work without loss of generality in an affine chart when enumerating all possible combinations of points in general position. In steps 3 and 4, we construct precisely those points with the Galois orbit structure of interest. Hence, the algorithm either finds a set of points in general position with the desired orbit structure or prove the non-existence of such points. 
\end{proof}

Using the above algorithm and the lower bound in the previous section, one obtains the statements in Theorem \ref{main theo 1} for the following types: 2, 3, 4, 5, 6, 7, 9, 10, 11, 12, 13, 18, 20, 21, 22, 23, 37, 39, 40, and their Bertini twists. 

We now treat the types 1, 35 and 63. Types 1 and 63 were already considered in \cite[Section 5.2]{Loughran19}. For the sake of completeness, we briefly recall the ideas of the proofs here. Type 1 is special in the sense that it does not exist over larger finite fields (compared with the other types). Ruling out the existence over $\bF_q$ for $q=11$, 13, 16 and $17$ is computationally quite heavy. Fortunately, this kind of surface was studied in the literature and certain classification results are available. 
\begin{defi}\label{Definition split and full}
    Let $X$ be a del Pezzo surface over a finite field $\bF_q$. Then $X$ is called \emph{split} if all its $(-1)$-curves are defined over $\bF_q$. Moreover, a split del Pezzo surface is called \emph{full} if all $\bF_q$-rational points lie on the $(-1)$-curves.
\end{defi}

For degree $\geq 3$ these surfaces were studied by Hirschfeld already in the 1980's and earlier, cf.\ \cite{Hirschfelddelpezzofinitefields} and also \cite[Theorem 20.3.9 and 20.3.10]{Hirschfeldfiniteprojspaces}. In particular, he showed that for $q \leq 8$ all split cubic surfaces are full over $\bF_q$ and that for $q \geq 9$ there always exists a split cubic surface that is not full. 

More recently, Knecht and Reyes have generalized his results to degree 2 del Pezzo surfaces, see \cite{knechtfulldegreetwo}. They prove that all split degree 2 del Pezzo surfaces over $\bF_{9}$, $\bF_{11}$ and $\bF_{13}$ are full, see \cite[Corollary 4.4]{knechtfulldegreetwo}. Furthermore, as they remark in their proof of Theorem 4.5, it follows from further results by Kaplan \cite{Kaplan} that there is a unique split but not full degree 2 del Pezzo surface over $\bF_{17}$. This surface is branched over the so-called Kuwata curve $C_{234}$, see \cite{Kuwata}, and using the results in Kuwata's paper, it is easy to see that all rational points lie on the ramification curve. For $q=16$ and greater $q$, split degree 1 del Pezzo surfaces exist, which can be verified using Algorithm 1.  

Next, we discuss the types 35 and 63, which do not contain a rational $(-1)$-curve; hence, we do not have lower bounds on their existence over a finite field.

For type 63, which can be realized as a blow-up of a degree 8 point in $\bP^2_{\bF_q}$, we recall the very nice idea provided in \cite[Section 5.2]{Loughran19}: Consider a normal basis $\alpha_i$ of the extension $\bF_{q^8}/\bF_q$ and the eight points $[1:\alpha_i:\alpha_i^3]$. The idea is to show that these points are always in general position using Lemma \ref{Lemma matrix cond gen pos} and the properties of a normal basis. We refer to \cite[Section 5.2]{Loughran19} for these computations. 

We are now left with type 35, which can be realized as a blow-up of two degree 4 points in $\bP^2_{\bF_q}$. A direct application of the arguments as for type 63 seems not possible since we blowup two and not just one point. Instead, we use the following idea: Let $\alpha_i$ be a normal basis for the extension $\bF_{q^4}/\bF_q$. Consider the eight points $[1:\alpha_i:\alpha_i^3]$ and $[1:\alpha_i+x:\alpha_i^3]$ for some $x \in \bF_q^\times$. The idea is to show that for $q$ large enough there always has to exist a $x \in \bF_q^\times$ such that the above eight points are in general position. We will do so by bounding the degrees of the determinants considered as polynomials in $x$. Doing this carefully enough yields a moderate bound. 

\begin{lem}\label{Lemma type 35}
    A del Pezzo surface of degree $1$ and type $35$ exists over $\bF_q$ for all $q$.
\end{lem}

\begin{proof}
    Let $\alpha_i$ for $i=1,2,3,4$ be a normal basis for the extension $\bF_{q^4}/\bF_q$. One checks that the points $P_i=[1:\alpha_i:\alpha_i^3]$ and $Q_i=[1:\alpha_i+x:\alpha_i^3]$ define two degree 4 points for $x \in \bF_q^\times$. Applying Lemma \ref{Lemma matrix cond gen pos} to these points and expanding all determinants occurring there we get certain polynomials in $x$. Note that if the determinant of the $10 \times 10$ matrix, obtained from the $11 \times10$ matrix in Lemma \ref{Lemma matrix cond gen pos} by removing the last row, does not vanish, the $11 \times 10$ matrix has a trivial kernel. 
    
    In the following we will group these polynomials in a sensible way and analyze their degrees to get a lower bound for $q$ since if $q$ is taken to be larger then the sum of these degrees, the existence of a $x$ such that the above eight points are in general position is guaranteed.
    
    From the first condition of Lemma \ref{Lemma matrix cond gen pos} we get $\binom{8}{3}=56$ polynomials in $x$. There are four polynomials where each point is one of the $P_i$'s. These clearly have degree 0 in $x$. Note that in this case the determinant does not vanish since
    \[\left|\begin{pmatrix}
        1 & 1 & 1\\
        \alpha_i & \alpha_j & \alpha_r\\
        \alpha_i^3 & \alpha_j^3 & \alpha_r^3
    \end{pmatrix}\right| = (\alpha_i-\alpha_j)(\alpha_i-\alpha_r)(\alpha_j-\alpha_r)(\alpha_i+\alpha_j+\alpha_r)\]
    which is for pairwise different $i,j,r$ always non-zero since the $\alpha_i$'s are a normal basis. 
    
    There are 24 polynomials where two of the three points are from the $P_i$'s, 24 where two of the three points are from the $Q_i$'s and again 4 where all three points are from the $Q_i$'s. Before we analyze the degrees of those polynomials we make the following observation:
    
    If $A=(P_i|P_j|Q_k)$ and $B=(\varphi(P_i)|\varphi(P_j)|\varphi(P_k))$ for pairwise different $i,j,r$, and where $\varphi$ is the Frobenius homomorphism are two matrices, then the determinants are Galois conjugated polynomials over $\bF_q$ and hence have the same zeros in $\bF_q$. Consequently, we do not have to count all degrees but just for one representative from each orbit. Hence, we have to check the degrees of $6+6+1=13$ polynomials. The first six are of the form
    \[\left| \begin{pmatrix}
        1 & 1 & 1\\
        \alpha_i & \alpha_j+x & \alpha_r\\
        \alpha_i^3 & \alpha_j^3 & \alpha_r^3
    \end{pmatrix}\right|= x(\alpha_r^3-\alpha_i^3)+(\alpha_i-\alpha_j)(\alpha_i-\alpha_r)(\alpha_j-\alpha_r)(\alpha_i+\alpha_j+\alpha_r)\]
    for pairwise different $i,j,r$ and accordingly linear. The next six are of the form
    \[\left| \begin{pmatrix}
        1 & 1& 1\\
        \alpha_i & \alpha_j+x & \alpha_r+x\\
        \alpha_i^3 & \alpha_j^3 & \alpha_r^3
    \end{pmatrix}\right| = x(\alpha_r^3-\alpha_j^3)+(\alpha_i-\alpha_j)(\alpha_i-\alpha_r)(\alpha_j-\alpha_r)(\alpha_i+\alpha_j+\alpha_r)\]
    for pairwise different $i,j,r$ and also linear. The last one is of the form 
    \[\left| \begin{pmatrix}
        1 & 1 & 1\\
        \alpha_i+x & \alpha_j+x & \alpha_r+x\\\alpha_i^3 & \alpha_j^3 & \alpha_r^3
    \end{pmatrix}\right| = \left|\begin{pmatrix}
        1 & 1 & 1\\
        \alpha_i & \alpha_j & \alpha_r\\
        \alpha_i^3 & \alpha_j^3 & \alpha_r^3
    \end{pmatrix}\right|\]
    for pairwise different $i,j,r$ and so has degree 0.
    
    Next, we have to consider the $\binom{8}{6}= 28$ matrices from the second condition. We have six matrices with all four $P_i$'s and two of the $Q_i$'s, 14 with three of the $P_i$'s and three of the $Q_i$'s and again six with two of the $P_i$'s and all $Q_i$'s. Note, that we can also collect orbits as above. We have two orbits for the first six matrices, three for the 12 matrices and again two for the last six matrices. One of the first six matrices is a matrix $A$ of the form where the first four rows are
    \[ \begin{pmatrix}
        1 & \alpha_i^2 & \alpha_i^6 & \alpha_i & \alpha_i^3 & \alpha_i^5 
    \end{pmatrix}\]
    for $i=1,2,3,4$ and the last two are
    \[\begin{pmatrix}
        1 & (\alpha_k+x)^2 & \alpha_k^3 & \alpha_k+x & \alpha_k^3 & \alpha_k^3(\alpha_k+x)\\
        1 & (\alpha_l+x)^2 & \alpha_l^3 & \alpha_l+x & \alpha_l^3 & \alpha_l^3(\alpha_l+x)
    \end{pmatrix}\]
    for pairwise different $l,k \in \{1,2,3,4\}$. By a closer look at the Leibniz formula
    \[\det(A) = \sum_{\sigma \in \fS_6}\left(\mathrm{sgn}(\sigma)\prod_{i=1}^6a_{i\sigma(i)}\right)\]
    one sees that an upper bound for the degree of the polynomial is 3. Similarly, one sees that the bound for the other two kinds of matrices is 4. 
    
    Finally, we consider the $10 \times 10$ matrix obtained from removing the last row of the $11 \times 10$ matrix from the third condition. We have two orbits to study here. But again a closer look to the Leibniz formula yields that a degree bounds are 8 and 9. Summing up all degrees we get 55. Hence for all $q\geq 59 $ there exist two degree 4 points in general position in $\bP^2_{\bF_q}$. Algorithm 1 proves the existence for all $2 \leq q \leq 53$.  
\end{proof}

\subsection{Sextics in \texorpdfstring{$\bP_{\bF_q}(1,1,2,3)$}{P\_F\_q(1,1,2,3)}}\label{sec: sextics}

We saw in \S\! \ref{sec: antican systems} that every del Pezzo surface of degree 1 over a field $k$ is a smooth sextic in $\bP_k(1,1,2,3)$ and vice versa. The singular points in $\bP_k(1,1,2,3)$ are $[0:0:1:0]$ and $[0:0:0:1]$. Hence, a smooth sextic in $\bP_k(1,1,2,3)$ is always given by the zero locus of a polynomial of the form

\[w^2+z^3+wzf_1(x,y)+z^2f_2(x,y)+wf_3(x,y)+zf_4(x,y)+f_6(x,y),\]
where $f_i\in k[x,y]$ is homogeneous of degree $i$. In what follows, the variables $x$ and $y$ will always have weight 1, the variable $z$ weight 2 and $w$ weight 3.

Due to \cite[Theorem 1.2]{alvaradoexplicit-1}, we have an explicit description of the $(-1)$-curves:
\begin{theo}
    Let $X$ be a del Pezzo surface of degree $1$ over a perfect field $k$ given by $V(f(x,y,z,w)) \subseteq \bP_k(1,1,2,3)$. Let $Q(x,y) \in \overline{k}[x,y]$ be homogeneous of degree $2$ and $C(x,y) \in \overline{k}[x,y]$ be homogeneous of degree $3$. If $V(z-Q(x,y),w-C(x,y))$ is a divisor on $\overline{X}$, then it is a $(-1)$-curve. Conversely, every $(-1)$-curve on $\overline{X}$  is of this form. 
\end{theo}

\begin{defi}\label{Definition ot sequ}
    Let $X$ be a del Pezzo surface of degree $1$ over $\bF_q$. Denote by $a_n$ be the number of $\bF_{q^n}$-rational points on $X$, and by $b_n$ the number of orbits of $(-1)$-curves of length $n$ of $X$. We call the sequence
    \[(a_n,b_n)_{n \in \mathbb{N}}\]
    the \emph{Orbit-Trace-sequence}. 
\end{defi}

\begin{rema}
    We computed this sequence for all 112 types of degree 1 del Pezzo surfaces, cf.\ \cite{Karras}. In most cases, it is enough to know the sequence up to $n=2$ or $3$. Unfortunately, one has to compute the sequence up to $a_7$ to distinguish the types 40 (and its Bertini twist type 95) and 61 (and its Bertini twist 79). Type 40 was considered in the last section. To find an example of type 79 over $\bF_3$, we use a refined point count, and this type is treated in an extra Magma code, cf.\ \cite{Karras}. 
\end{rema}

Given a smooth sextic, one can compute this sequence as indicated by the algorithm below.

\noindent \textbf{Algorithm 2:} Points and $(-1)$-curves on Sextices in $\bP_{\bF_q}(1,1,2,3)$ 

\noindent \textbf{Input:} A smooth sextic $X=V(f(x,y,z,w))$ in $\bP_{\bF_q}(1,1,2,3)$ and a sequence of non-negative integers $n_1,\dots,n_r$

\noindent\textbf{Output:} The number of $\bF_q$-points on $X$ and for each $n_i$ the number of $(-1)$-curves defined over $\bF_{q^{n_i}}$ 
\begin{enumerate}
    \item Count the number of solutions of $f$ up to projective equivalence
    \item For each $n_i$ count the number of pairs $(s(x,y),t(x,y))$, where $s$ (resp. $t$) is a homogeneous degree 2 (resp. 3) form in $\bF_{q^{n_i}}[x,y]$ such that $f(x,y,s(x,y),t(x,y))=0$ 
\end{enumerate}

In our Magma implementation, we use a slightly different point count as implemented in \cite{Loughran19}. In that way our time complexity is $\cO(q^3)$, instead of $\cO(q^4)$. To determine the number of orbits as indicated in Algorithm 2, we do not have to loop over all $s$ and $t$. Instead, we do the following: For each quadratic polynomial $s$ the expression $f(x,y,s(x,y),T) \in \bF_q[x,y][T]$ defines a quadratic polynomial in $T$ with coefficients in the ring $\bF_q[x,y]$. Now, we just have to compute the number of solutions of this polynomial in $T$. Magma can compute the factorization of such a polynomial and hence we just have to count the number of different factors.

In the remaining part of this section, we study the equations in more detail. The first result is the following proposition, which is well-know. For the sake of completeness, we recall the simple proof. 

\begin{prop}\label{Proposition quadratic forms finite fields}
    Let $X$ be a del Pezzo surface of degree $1$ over a field $k$. Then, depending on $\mathrm{char}(k)$, we have the following normal forms for $X$, where $f_i \in k[x,y]$ is homogeneous of degree $i$:
    \[X \cong \begin{cases}
        V(w^2+z^3+zf_4(x,y)+f_6(x,y)), \: \text{if} \: \mathrm{char}(k) \ne 2,3, \\
        V(w^2+z^3+z^2f_2(x,y)+zf_4(x,y)+f_6(x,y)), \: \text{if} \: \mathrm{char}(k) \ne 2,\\
        V(w^2+z^3+wzf_1(x,y)+wf_3(x,y)+zf_4(x,y)+f_6(x,y)), \: \text{if} \: \mathrm{char}(k) \ne 3. \\
    \end{cases}\]
    
\end{prop}

\begin{proof}
    The transformation 
    \[w \mapsto w - \frac{zf_1(x,y)+f_3(x,y)}{2}\]
    works in $\mathrm{char}(k)\ne 2$ and kills $f_1(x,y)$ and $f_3(x,y)$. The transformation 
    \[z \mapsto z-\frac{f_2(x,y)}{3}\]
    works in $\mathrm{char}(k)\ne 3$ and kills $f_2(x,y)$. 
\end{proof}

These normal forms are already very useful to produce examples of certain types. In Magma, we randomly generate equations over a fixed finite field $\bF_q$ and compute the Orbit-Trace-sequence, or less data if possible, as indicated above. A list of the obtained examples and their types is provided in \cite{Karras}. They can be verified using our Magma code. In combination with the lower bound from \S\! \ref{sec: lower bound}, this already proves the existence claims in Theorem \ref{main theo 1} and \ref{main theo 2}. We next deal with the non-existence claims. 

As already remarked in \S\! \ref{sec: lower bound}, the non-existence for certain types already follows from the non-existence of the corresponding degree 2 del Pezzo surface. For the remaining cases, we still have that many types do not exist over $\bF_2$ and $\bF_3$. Furthermore, it turns out that type 19 does not exist over $\bF_4$ and type 8 does not exist over $\bF_7$. The number of equations in the above normal forms over these fields are the following:
\begin{enumerate}
    \item $\bF_2$: $2^{18}= 262144$,
    \item $\bF_3$: $3^{15}=14348907$,
    \item $\bF_4$: $4^{18}= 68719476736$,
    \item $\bF_7$: $7^{12}= 13841287201$.
\end{enumerate}

Our Magma code needs only seconds or minutes (depending on how many Galois orbits and rational points over field extensions one computes for each surface) to loop over all $2^{18}$ surfaces over $\bF_2$. A loop over the $3^{15}$ surfaces over $\bF_3$ is also possible but needs several hours. Since it is in this case quite easy to derive better normal forms (in the sense that the total number of equations is much less), we will do so. 

To derive better normal forms over $\bF_3$ (resp. $\bF_4$), we need the classification of binary quadratic (resp. cubic) forms over these fields.

\begin{prop}\label{Proposition class quadr forms odd q}
    Let $\bF_q$ be a finite field with $q$ odd. Let $0 \ne f \in \bF_q[x_1,\dots,x_n]$ be a homogeneous degree $2$ form. Let $a \in \bF_q$ be a non-square. Then $f$ is $\mathrm{GL}(n,\bF_q)$-equivalent to one of the following forms: $x_1^2$, $ax_1^2$, $x_1^2+x_2^2$, $x_1^2+ax_2^2$, $\dots$, $x_1^2+\dots+x_n^2$, $x_1^2+\dots+ax_n^2$. 
\end{prop}

\begin{proof}
    By \cite[Theorem 6.21]{Niederreiterfinitefields} $f$ is equivalent to a diagonal quadratic form, i.e. we can write $f = a_1x_1^2+\dots+a_nx_n^2$. Without loss of generality, we may assume that $f$ is non-degenerate and hence all $a_i$ are non-zero. By \cite[Remark 6.25]{Niederreiterfinitefields}, the equation $a_1x_1^2+a_2x_2^2=1$ always has a solution and hence by \cite[Lemma 6.20]{Niederreiterfinitefields} this form is equivalent to $x_1^2+bx_2^2$ for some $b \in \bF_q$. Inductively, this implies the claim.  
\end{proof}

\begin{prop}
   Let $\bF_4=\bF_2(\alpha)$ with $\alpha^2+\alpha+1=0$ be the field with four elements. Let $0 \ne f \in \bF_4[x,y]$ be a cubic form. Then $f$ is $\mathrm{GL}(2,\bF_4)$-equivalent to one of the following $13$ forms: $x^3$, $\alpha x^3$, $\alpha^2x^3$, $xy^2$, $x^3+y^3$, $\alpha(x^3+y^3)$, $\alpha^2(x^3+y^3)$, $x(x^2+xy+\alpha y^2)$, $\alpha x(x^2+xy+\alpha y^2)$, $\alpha^2 x(x^2+xy+\alpha y^2)$, $x^3+\alpha y^3$, $\alpha(x^3+\alpha y^3)$, $\alpha^2(x^3+\alpha y^3)$.   
\end{prop}

\begin{proof}
  From the classification of binary quadratic forms over $\bF_q$ for even $q$, cf.\ \cite[Theorem 6.30]{Niederreiterfinitefields}, it follows that up to $\mathrm{GL}(2,\bF_4)$-equivalence, there is exactly one irreducible quadratic form. One easily verifies that $x^2+xy+\alpha y^2$ is a representative. To identify all irreducible cubic forms, it is helpful to consider this geometrically: Up to scaling, such a form corresponds to three Galois conjugated points in $\bP^1_{\bF_{64}}$. Now, the number of points in $\bP^1_{\bF_{64}}$ strictly defined over $\bF_{64}$ is 60. Denote the set of these points by $\cP$. The group $\mathrm{PGL}(2,\bF_4)$ has order 60 and acts on $\cP$. We claim that each element $[x:1] \in \cP$ has a trivial stabilizer. Indeed, if a matrix $(a_{ij})\in \mathrm{PGL}(2,\bF_4)$ fixes $[x:1]$, there exists a $\theta \in \bF_{64}^\times$ such that 
  \[\begin{pmatrix}
      a_{11} & a_{12}\\
      a_{21} & a_{22}
  \end{pmatrix}\begin{pmatrix}
      x\\
      1
  \end{pmatrix} = \begin{pmatrix}
      \theta x\\
      \theta
  \end{pmatrix}.\]
  But then $x$ satisfies a quadratic relation over $\bF_4$, which contradicts the assumption that $x$ is strictly defined over $\bF_{64}$. It follows that this action has precisely one orbit by the orbit-stabilizer formula. Let $\cP'$ be the set $\cP$ modulo the Galois action on $\cP$. Then $\cP'$ has cardinality 20. Observe that the action of $\mathrm{PGL}(2,\bF_4)$ on $\cP$ induces a well-defined action on $\cP'$, which is again transitive. Again, by the orbit-stabilizer formula, each stabilizer has cardinality 3. 

  Consequently, since $\bF_4$ acts trivially on cubic forms, there are three irreducible cubic forms up to $\mathrm{GL}(2,\bF_4)$-equivalence. One can now consider the different factorizations of a cubic form and obtain in that way the above list. 
\end{proof}

The above propositions can be used to derive normal forms for degree 1 del Pezzo surfaces. The precise equations are listed in the appendix. They are a result of the above transformations in $x$, $y$ and a further shift in $z$. Over $\bF_3$ this reduces the number of equations to $5 \cdot 3^9= 98415$. Our code loops over these surfaces in seconds. Over $\bF_4$ the number is reduced to $13 \cdot 4^9 +4^8+4^{12}= 20250624$. Our code needs roughly one day for a complete loop. 

We now come to equations over $\bF_7$. A similar strategy as above would mean to classify binary quartic forms over $\bF_7$ up to $\mathrm{GL}(2,\bF_7)$-equivalence. This seems to be more complicated and instead, we are satisfied with normal forms where we have certain redundancies.

\begin{prop}
    Let $0\ne f=ax^4+bx^3y+cx^2y^2+dxy^3+ey^4 \in \bF_7[x,y]$ be a binary quartic form. Then $f$ is $\mathrm{GL}(2,\bF_7)$-equivalent to one of the following normal forms:
    \begin{enumerate}
        \item $ax^4+bx^3y+cx^2y^2+dxy^3+ey^4$ with $a,c \in \{1,3\}$, $b=0$, $d \in \{0,1,2,3\}$ and $e\in \bF_7$,
        \item $ax^4+bx^3y+cx^2y^2+dxy^3+ey^4$ with $a\in \{1,3\}$, $b=c=0$, $d\in \{0,1,2,3\}$ and $e \in \{0,1,3\}$.
    \end{enumerate}
\end{prop}

\begin{proof}
    Without loss of generality, we may assume that $a \ne 0$. By scaling $x$, we always obtain $a \in \{1,3\}$. Using a shift of the form $x \mapsto x +\lambda y$ for a certain $\lambda \in \bF_7$, we may also assume that $b=0$. If $c \ne 0$, it can be scaled by $y$ to 1 or 3, and the coefficient $d$ can at least be ensured to have a positive sign. The coefficient $e$ has to be arbitrary in this case. If $c=0$, we can scale $e$ using $y$ and again adjust the sign of $d$. 
\end{proof}

Since for the normal forms over $\bF_7$ we have no more changes in $z$ and $w$ they are just obtained by plugging in the above forms. These are now $136 \cdot 7^7 = 112001848$ equations. Our code needs roughly 3 days for a complete loop. 

Now all non-existence claims in Theorem \ref{main theo 1} and \ref{main theo 2} can be verified using the code provided in \cite{Karras}.

\subsection{Open cases}\label{sec: open cases}

The remaining cases may be divided into two classes: the minimal and the non-minimal surfaces. A nice surface over a field $k$ is called minimal if any birational morphism to any nice surface $X'$ is actually an isomorphism over $k$. It is not too difficult to see that a surface $X$ over a field $k$ is minimal if and only if there does not exist a Galois invariant set of pairwise disjoint $(-1)$-curves on $\overline{X}$. 

The non-minimal types can be obtained by blowing up certain types of degree 4 and 5 del Pezzo surfaces. It should be possible to derive a similar lower bound as we did in \S\! \ref{sec: lower bound} also in this case. 

The minimal types form the actual hard part of the problem. It is well known that a minimal del Pezzo surface of degree 1 or 2 either has arithmetic Picard rank 1 or 2, where in the latter case it has the structure of a conic bundle. In case of degree 2 del Pezzo surfaces, it turns out that any minimal surface has either the structure of a conic bundle or its Geiser twist is non-minimal. This fails for the degree 1 case and in particular for most of the types listed below. We refer to \cite{trepalin20}, \cite{kaya2025inversecurveproblemsdel}, and \cite{trip} for more details. In Table \ref{Table nonmin}, we list the open non-minimal types with their (also open) Bertini twist. In Table \ref{Table min}, we list the minimal open types, whose Bertini twist is also open and minimal.     Using the recent results obtained by Trip \cite{trip}, and using Theorem \ref{main theo 1} and \ref{main theo 2}, one also obtains a complete solutions for the types 66, 75, 76, 77, and their Bertini twist, which are marked with a $\ast$ in Tables \ref{Table nonmin} and \ref{Table min}.

\section{Appendix}\label{sec: appendix}

\renewcommand{\thetable}{\Roman{table}} 
\setcounter{table}{0} 

In the following, we collect some necessary data used in this paper. Tables \ref{Table WE7} and \ref{Table WE8} contain information related to cyclic conjugacy classes of the groups $W(E_7)$ and $W(E_8)$. 

The first column contains the number, and the second column contains the Carter symbol of such a class in the order of the tables in \cite{Carter} (note that Urabe in \cite{Urabe} uses a different order). 

The third column contains the orbit type. The notation $a^bc^d\dots$ means that there are $b$ orbits of size $a$, $d$ orbits of size $c$ and so on. This information can be found in \cite{Urabe} (note that we have summarized the notation in case $a^ba^c$ to $a^{(b+c)}$ since this information is enough for our purposes and makes the tables clearer). 

The fourth column contains the eigenvalues of the characteristic polynomial of an element in the conjugacy class. The $n$th root of unity is denoted by $\zeta_n$ where $\zeta_4=i$. 

Observe that with $w \in W(E_n)$ we can associate two vector space homomorphisms: the one on the Euclidean space of dimension $n$ containing the root system $E_n$ and the one on $N_n\otimes \bR$ of dimension $n+1$. It is not too difficult to see that if the former has the characteristic polynomial $\varphi(x)$ the latter has the characteristic polynomial $(x-1)\varphi(x)$. In the tables we list the eigenvalues of $\varphi(x)$. The polynomials can again be found in \cite{Urabe}. 

The fifth column contains the type of the Geiser and the Bertini twist. The last column contains the sum of the eigenvalues adding 1 so that we get the trace of the Frobenius action on the Picard group. 

Table \ref{Table H1} compares the values of $H^1(G,\mathrm{Pic}(\overline{X}))$ (see also Section 2.2) for certain types of del Pezzo surfaces of degrees 1 and 2. This information can again be found in \cite{Urabe}. 

Tables \ref{Table nonmin} and \ref{Table min} list the open non-minimal and minimal cases. 

We have found a small typo in Urabe's tables: In Table \ref{Table WE8}, type 9 has 116 orbits of length 2 (and not 115 as it is written in Urabe's tables). 

\begin{table}[H]
\footnotesize
    \centering
    \caption{$W(E_7)$}
    \label{Table WE7}
    \begin{tabular}{c c c c c c}
    \toprule
    Type & Carter Symbol & Orbit Type & Eigenvalues & Geiser & $a(X)$ \\ 
    \midrule
    1 & $\emptyset$ & $1^{56}$ & 1, 1, 1, 1, 1, 1, 1 & 49 & 8\\
    2 & $A_1$ & $1^{32}2^{12}$ & 1, 1, 1, 1, 1, 1, $-1$ & 31 & 6\\
    3 & $A_1^2$ & $1^{16}2^{20}$ & 1, 1, 1, 1, 1, $-1$, $-1$ & 18 & 4\\
    4 & $A_2$ & $1^{20}3^{12}$ & 1, 1, 1, 1, 1, $\zeta_3$, $\zeta_3^2$ & 53 & 5\\
    5 & $(A_1^3)'$ & $2^{28}$ & 1, 1, 1, 1, $-1$, $-1$, $-1$ & 10 & 2\\
    6 & $(A_1^3)''$ & $1^82^{24}$ & 1, 1, 1, 1, $-1$, $-1$, $-1$ & 9 & 2\\
    7 & $A_2 \times A_1$ & $1^82^63^86^2$ & 1, 1, 1, 1, $\zeta_3$, $\zeta_3^2$, $-1$ & 40 & 3\\
    8 & $A_3$ & $1^{12}2^24^{10}$ & 1, 1, 1, 1, $i$, $-i$, $-1$ & 33 & 4\\
    9 & $(A_1^4)'$ & $2^{28}$ & 1, 1, 1, $-1$, $-1$, $-1$, $-1$ & 6 & 0\\
    10 & $(A_1^4)''$ & $1^82^{24}$ & 1, 1, 1, $-1$, $-1$, $-1$, $-1$ & 5 & 0\\
    11 & $A_2\times A_1^2$ & $1^42^83^46^4$ & 1, 1, 1, $\zeta_3$, $\zeta_3^2$, $-1$, $-1$ & 27 & 1\\
    12 & $A_2^2$ & $1^23^{18}$ & 1, 1, 1, $\zeta_3$, $\zeta_3^2$, $\zeta_3$, $\zeta_3^2$ & 55 & 2\\
    13 & $(A_3\times A_1)'$ & $2^84^{10}$ & 1, 1, 1, $i$, $-i$, $-1$, $-1$ & 22 & 2\\
    14 & $(A_3\times A_1)''$ & $1^42^64^{10}$ & 1, 1, 1, $i$, $-i$, $-1$, $-1$ & 21 & 2\\
    15 & $A_4$ & $1^65^{10}$ & 1, 1, 1, $\zeta_5$, $\zeta_5^2$, $\zeta_5^3$, $\zeta_5^4$ & 54 & 3\\
    16 & $D_4$ & $1^82^66^6$ & 1, 1, 1, $-\zeta_3$,$-\zeta_3^2$,$-1$,$-1$ & 19 & 3\\
    17 & $D_4(a_1)$ & $1^84^{12}$ & 1, 1, 1, $i$, $i$, $-i$, $-i$ & 50 & 4\\
    18 & $A_1^5$ & $2^{28}$ & 1, 1, $-1$, $-1$, $-1$, $-1$, $-1$ & 3 & $-2$\\
    19 & $A_2\times A_1^3$ & $2^{10}6^6$ & 1, 1, $\zeta_3$, $\zeta_3^2$, $-1$, $-1$, $-1$ & 16 & $-1$\\
    20 & $A_2^2\times A_1$ & $1^23^{10}6^4$ & 1, 1, $\zeta_3$, $\zeta_3^2$, $\zeta_3$, $\zeta_3^2$, $-1$ & 45 & 0\\
    \bottomrule
    \end{tabular}
\end{table}

\newpage

\begin{table}[H]
\footnotesize
    \centering
    \ContinuedFloat
    \caption[]{$W(E_7)$ (continued) }
    \begin{tabular}{c c c c c c}
    \toprule
    Type & Carter Symbol & Orbit Type & Eigenvalues & Geiser & $a(X)$  \\ 
    \midrule
    21 & $(A_3\times A_1^2)'$ & $2^84^{10}$ & 1, 1, $i$, $-i$, $-1$, $-1$, $-1$ & 14 & 0\\
    22 & $(A_3\times A_1^2)''$ & $1^42^64^{10}$ & 1, 1, $i$, $-i$, $-1$, $-1$, $-1$ & 13 & 0\\
    23 & $A_3\times A_2$ & $2^23^44^412^2$ & 1, 1, $i$, $-i$, $-1$, $\zeta_3$, $\zeta_3^2$ & 42 & 1\\
    24 & $A_4 \times A_1$ & $1^22^25^610^2$ & 1, 1, $\zeta_5$, $\zeta_5^2$, $\zeta_5^3$, $\zeta_5^4$, $-1$ & 43 & 1\\
    25 & $(A_5)'$ & $2^16^9$ & 1, 1, $-\zeta_3^2$, $-\zeta_3$, $-1$, $\zeta_3^2$, $\zeta_3$ & 38 & 2\\
    26 & $(A_5)''$ & $1^23^26^8$ & 1, 1, $-\zeta_3^2$, $-\zeta_3$, $-1$, $\zeta_3^2$, $\zeta_3$ & 37 & 2\\
    27 & $D_4 \times A_1$ & $2^{10}6^6$ & 1, 1, $-\zeta_3$, $-\zeta_3^2$, $-1$, $-1$, $-1$ & 11 & 1\\
    28 & $D_4(a_1)\times A_1$ & $2^44^{12}$ & 1, 1, $i$, $i$, $-i$, $-i$, $-1$ & 35 & 2\\
    29 & $D_5$ & $1^42^28^6$ & 1, 1, $\zeta_8$, $\zeta_8^3$, $\zeta_8^5$, $\zeta_8^7$, $-1$ & 41 & 2\\
    30 & $D_5(a_1)$ & $1^44^46^212^2$ & 1, 1, $-\zeta_3$, $-\zeta_3^2$, $i$, $-i$, $-1$ & 34 & 3\\
    31 & $A_1^6$ & $2^{28}$ & 1, $-1$, $-1$, $-1$, $-1$, $-1$, $-1$ & 2 & $-4$\\
    32 & $A_2^3$ & $1^23^{18}$ & 1, $\zeta_3$, $\zeta_3^2$, $\zeta_3$, $\zeta_3^2$, $\zeta_3$, $\zeta_3^2$ & 60 & $-1$\\
    33 & $A_3\times A_1^3$ & $2^84^{10}$ & 1, $i$, $-i$, $-1$, $-1$, $-1$, $-1$ & 8 & $-2$\\
    34 & $A_3\times A_2\times A_1$ & $2^24^46^212^2$ & 1, $i$, $-i$, $\zeta_3$, $\zeta_3^2$, $-1$, $-1$ & 30 & $-1$\\
    35 & $A_3^2$ & $2^44^{12}$ & 1, $i$, $i$, $-i$, $-i$, $-1$, $-1$ & 28 & 0\\
    36 & $A_4\times A_2$ & $3^25^415^2$ & 1, $\zeta_5$, $\zeta_5^2$, $\zeta_5^3$, $\zeta_5^4$, $\zeta_3$, $\zeta_3^2$ & 59 & 0\\
    37 & $(A_5\times A_1)'$ & $2^16^9$ & 1, $-\zeta_3^2$, $-\zeta_3$, $\zeta_3^2$, $\zeta_3$, $-1$, $-1$ & 26 & 0\\
    38 & $(A_5\times A_1)''$ & $1^23^26^8$ & 1, $-\zeta_3^2$, $-\zeta_3$, $\zeta_3^2$, $\zeta_3$, $-1$, $-1$ & 25 & 0\\
    39 & $A_6$ & $7^8$ & 1, $\zeta_7$, $\zeta_7^2$, $\zeta_7^3$, $\zeta_7^4$, $\zeta_7^5$, $\zeta_7^6$ & 57 & 1\\
    40 & $D_4\times A_1^2$ & $2^{10}6^6$ & 1, $-\zeta_3$, $-\zeta_3^2$, $-1$, $-1$, $-1$, $-1$ & 7 & $-1$\\
    41 & $D_5 \times A_1$ & $2^48^6$ & 1, $\zeta_8$, $\zeta_8^3$, $\zeta_8^5$, $\zeta_8^7$, $-1$, $-1$ & 29 & 0\\
    42 & $D_5(a_1)\times A_1$ & $2^24^46^212^2$ & 1, $-\zeta_3$, $-\zeta_3^2$, $i$, $-i$, $-1$, $-1$ & 23 & 1\\
    43 & $D_6$ & $2^310^5$ & 1, $-1$, $-1$, $-\zeta_5$, $-\zeta_5^2$, $-\zeta_5^3$, $-\zeta_5^4$ & 24 & 1\\
    44 & $D_6(a_1)$ & $4^28^6$ & 1, $i$, $-i$, $\zeta_8$, $\zeta_8^3$, $\zeta_8^5$, $\zeta_8^7$ & 52 & 2\\
    45 & $D_6(a_2)$ & $2^16^9$ & 1, $-\zeta_3$, $-\zeta_3^2$, $-\zeta_3$, $-\zeta_3^2$, $-1$, $-1$ & 20 & 2\\
    46 & $E_6$ & $1^23^212^4$ & 1, $\zeta_3$, $\zeta_3^2$, $\zeta_{12}$, $\zeta_{12}^5$, $\zeta_{12}^7$, $\zeta_{12}^{11}$ & 58 & 1\\
    47 & $E_6(a_1)$ & $1^29^6$ & 1, $\zeta_9$, $\zeta_9^2$, $\zeta_9^4$, $\zeta_9^5$, $\zeta_9^7$, $\zeta_9^8$ & 56 & 2\\
    48 & $E_6(a_2)$ & $1^23^26^8$ & 1, $\zeta_3$, $\zeta_3^2$, $-\zeta_3$, $-\zeta_3$, $-\zeta_3^2$, $-\zeta_3^2$,  & 51 & 3\\
    49 & $A_1^7$ & $2^{28}$ & $-1$, $-1$, $-1$, $-1$, $-1$, $-1$, $-1$ & 1 & $-6$\\
    50 & $A_3^2\times A_1$ & $2^44^{12}$ & $-1$, $-1$, $-1$, $i$, $i$, $-i$, $-i$ & 17 & $-2$\\
    51 & $A_5\times A_2$ & $2^16^9$ & $-1$, $\zeta_3$, $\zeta_3$, $-\zeta_3$, $\zeta_3^2$, $\zeta_3^2$, $-\zeta_3^2$ & 48 & $-1$\\
    52 & $A_7$ & $4^28^6$ & $-1$, $i$, $-i$, $\zeta_8$, $\zeta_8^3$, $\zeta_8^5$, $\zeta_8^7$ & 44 & 0\\
    53 & $D_4 \times A_1^3$ & $2^{10}6^6$ & $-1$, $-1$, $-1$, $-1$, $-1$, $-\zeta_3$, $-\zeta_3^2$ & 4 & $-3$\\
    54 & $D_6\times A_1$ & $2^310^5$ & $-1$, $-1$, $-1$, $-\zeta_5$, $-\zeta_5^2$, $-\zeta_5^3$, $-\zeta_5^4$ & 15 & $-1$\\
    55 & $D_6(a_2)\times A_1$ & $2^16^9$ & $-1$, $-1$, $-1$, $-\zeta_3$, $-\zeta_3^2$, $-\zeta_3$, $-\zeta_3^2$ & 12 & 0\\
    56 & $E_7$ & $2^118^3$ & $-1$, $-\zeta_9$, $-\zeta_9^2$, $-\zeta_9^4$, $-\zeta_9^5$, $-\zeta_9^7$, $-\zeta_9^8$ & 47 & 0\\
    57 & $E_7(a_1)$ & $14^4$ & $-1$, $-\zeta_7$, $-\zeta_7^2$, $-\zeta_7^3$, $-\zeta_7^4$, $-\zeta_7^5$, $-\zeta_7^6$ & 39 & 1\\
    58 & $E_7(a_2)$ & $2^16^112^4$ & $-1$, $-\zeta_3$, $-\zeta_3^2$, $-\zeta_{12}$, $-\zeta_{12}^5$, $-\zeta_{12}^7$, $-\zeta_{12}^{11}$ & 46 & 1\\
    59 & $E_7(a_3)$ & $6^110^230^1$ & $-1$, $-\zeta_3$, $-\zeta_3^2$, $-\zeta_5$, $-\zeta_5^2$, $-\zeta_5^3$, $-\zeta_5^4$ & 36 & 2\\
    60 & $E_7(a_4)$ & $2^16^9$ & $-1$, $-\zeta_3$, $-\zeta_3^2$, $-\zeta_3$, $-\zeta_3^2$, $-\zeta_3$, $-\zeta_3^2$ & 32 & 3\\
    \bottomrule
    \end{tabular}
\end{table}

\newpage

\begin{table}[H]
\footnotesize
    \centering
    \caption{$W(E_8)$}
    \label{Table WE8}
    \begin{tabular}{c c c c c c}
    \toprule 
    Type &Carter Symbol& Orbit Type & Eigenvalues & Bertini & $a(X)$ \\ 
    \midrule
    1  & $\emptyset$ & $1^{240}$ & 1, 1, 1, 1, 1, 1, 1, 1 & 83 &9\\
    2  & $A_1$ & $1^{126}2^{57}$ & 1, 1, 1, 1, 1, 1, 1, $-1$ & 52& 7\\
    3  & $A_1^2$ & $1^{60}2^{90}$ & 1, 1, 1, 1, 1, 1, $-1$, $-1$ & 28 &5\\
    4  & $A_2$& $1^{72}3^{56}$ & 1, 1, 1, 1, 1, 1, $\zeta_3$, $\zeta_3^2$ & 90 & 6\\
    5  & $A_1^3$& $1^{26}2^{107}$ & 1, 1, 1, 1, 1, $-1$, $-1$, $-1$ & 16&3\\
    6  & $A_2\times A_1$ & $1^{30}2^{21}3^{32}6^{12}$ & 1, 1, 1, 1, 1, $-1$, $\zeta_3$, $\zeta_3^2$ & 64& 4 \\
    7  & $A_3$& $1^{40}2^{10}4^{45}$ & 1, 1, 1, 1, 1, $-1$, $i$, $-i$ & 54 &5\\
    8  & $(A_1^4)'$ & $1^{24}2^{108}$ & 1, 1, 1, 1, $-1$, $-1$, $-1$, $-1$ & 8 &1\\
    9  & $(A_1^4)''$& $1^82^{116}$ & 1, 1, 1, 1, $-1$, $-1$, $-1$, $-1$ & 9 & 1\\
    10 & $A_2\times A_1^2$ & $1^{12}2^{30}3^{16}6^{20}$ & 1, 1, 1, 1, $-1$, $-1$, $\zeta_3$, $\zeta_3^2$ & 41&2\\
    11 & $A_2^2$ & $1^{12}3^{76}$ & 1, 1, 1, 1, $\zeta_3$, $\zeta_3$, $\zeta_3^2$, $\zeta_3^2$ & 91&3\\
    12 & $A_3\times A_1$ & $1^{14}2^{23}4^{45}$ & 1, 1, 1, 1, $-1$, $-1$, $i$, $-i$ & 32&3\\
    13 & $A_4$ & $1^{20}5^{44}$ & 1, 1, 1, 1, $\zeta_5$, $\zeta_5^2$, $\zeta_5^3$, $\zeta_5^4$ & 94&4 \\
    14 & $D_4$ & $1^{24}2^{24}6^{28}$ & 1, 1, 1, 1, $-1$, $-1$, $-\zeta_3$, $-\zeta_3^2$ & 29&4\\
    15 & $D_4(a_1)$ & $1^{24}4^{54}$ & 1, 1, 1, 1, $i$, $i$, $-i$, $-i$ & 85&5\\
    16 & $A_1^5$ & $1^62^{117}$ & 1, 1, 1, $-1$, $-1$, $-1$, $-1$, $-1$ & 5& $-1$\\
    17 & $A_2\times A_1^3$ & $1^22^{35}3^26^{24}$ & 1, 1, 1, $-1$, $-1$, $-1$, $\zeta_3$, $\zeta_3^2$ & 24 & 0\\
    18 & $A_2^2\times A_1$ & $1^62^33^{40}6^{18}$ & 1, 1, 1, $-1$, $\zeta_3$, $\zeta_3$, $\zeta_3^2$, $\zeta_3^2$ & 71 & 1\\
    19 & $(A_3\times A_1^2)'$ & $1^{12}2^{24}4^{45}$ & 1, 1, 1, $-1$, $-1$, $-1$, $i$, $-i$ & 19 &1\\
    20 & $(A_3\times A_1^2)''$ & $1^42^{28}4^{45}$ & 1, 1, 1, $-1$, $-1$, $-1$, $i$, $-i$ & 20 & 1\\
    21 & $A_3\times A_2$ & $1^42^43^{12}4^{15}6^212^{10}$ & 1, 1, 1, $-1$, $\zeta_3$, $\zeta_3^2$, $i$, $-i$ & 65 & 2\\
    22 & $A_4\times A_1$ & $1^62^75^{24}10^{10}$ & 1, 1, 1, $-1$, $\zeta_5$, $\zeta_5^2$, $\zeta_5^3$, $\zeta_5^4$ & 70 & 2\\
    23 & $A_5$ & $1^82^23^66^{35}$ & 1, 1, 1, $-1$, $\zeta_3$, $\zeta_3^2$, $-\zeta_3$, $-\zeta_3^2$ & 59 & 3\\
    24 & $D_4\times A_1$ & $1^62^{33}6^{28}$ & 1, 1, 1, $-1$, $-1$, $-1$, $-\zeta_3$, $-\zeta_3^2$ & 17 & 2\\
    25 & $D_4(a_1)\times A_1$ & $1^62^94^{54}$ & 1, 1, 1, $-1$, $i$, $i$, $-i$, $-i$ & 56 & 3\\
    26 & $D_5$ & $1^{12}2^68^{27}$ & 1, 1, 1, $-1$, $\zeta_8$, $\zeta_8^3$, $\zeta_8^5$, $\zeta_8^7$ & 67 & 3\\
    27 & $D_5(a_1)$ & $1^{12}4^{15}6^812^{10}$ & 1, 1, 1, $-1$, $-\zeta_3$, $-\zeta_3^2$, $i$, $-i$ & 55 & 4\\
    28 & $A_1^6$ & $1^42^{118}$ & 1, 1, $-1$, $-1$, $-1$, $-1$, $-1$, $-1$ & 3 & $-3$\\
    29 & $A_2\times A_1^4$  &  $2^{36}3^86^{24}$  &  1, 1, $-1$, $-1$, $-1$, $-1$, $\zeta_3$, $\zeta_3^2$  & 14 & $-2$   \\
    30 & $A_2^2\times A_1^2$ & $2^63^{20}6^{28}$ & 1, 1, $-1$, $-1$, $\zeta_3$, $\zeta_3$, $\zeta_3^2$, $\zeta_3^2$ & 48 & $-1$ \\
    31 & $A_2^3$ & $1^63^{78}$ & 1, 1, $\zeta_3$, $\zeta_3$, $\zeta_3$, $\zeta_3^2$, $\zeta_3^2$, $\zeta_3^2$ & 103 & 0 \\ 
    32 & $A_3 \times A_1^3$ & $1^22^{29}4^{45}$ & 1, 1, $-1$, $-1$, $-1$, $-1$, $i$, $-i$ & 12 & $-1$\\
    33 & $A_3\times A_2 \times A_1$ & $1^22^53^44^{15}6^612^6$ & 1, 1, $-1$, $-1$, $\zeta_3$, $\zeta_3^2$, $i$, $-i$ & 45 & 0\\
    34 & $(A_3^2)'$ & $1^42^{10}4^{54}$ & 1, 1, $-1$, $-1$, $i$, $i$, $-i$, $-i$ & 34 & 1\\
    35 & $(A_3^2)''$ & $2^44^{58}$  & 1, 1, $-1$, $-1$, $i$, $i$, $-i$, $-i$ & 35 & 1 \\
    36 & $A_4 \times A_1^2$ & $2^{10}5^{12}10^{16}$ &  1, 1, $-1$, $-1$, $\zeta_5$, $\zeta_5^2$, $\zeta_5^3$, $\zeta_5^4$ & 46 & 0 \\
    37 & $A_4\times A_2$ & $1^23^65^{14}15^{10}$ & 1, 1, $\zeta_3$, $\zeta_3^2$, $\zeta_5$, $\zeta_5^2$, $\zeta_5^3$, $\zeta_5^4$ & 97 & 1\\
    38 & $(A_5\times A_1)'$ & $1^62^33^66^{35}$ & 1, 1, $-1$, $-1$, $\zeta_3$, $\zeta_3^2$, $-\zeta_3$, $-\zeta_3^2$ & 38 & 1\\
    39 & $(A_5\times A_1)''$& $1^22^53^26^{37}$ & 1, 1, $-1$, $-1$, $\zeta_3$, $\zeta_3^2$, $-\zeta_3$, $-\zeta_3^2$ & 39 & 1\\
    40 & $A_6$ & $1^27^{34}$ & 1,1, $\zeta_7$, $\zeta_7^2$, $\zeta_7^3$, $\zeta_7^4$, $\zeta_7^5$, $\zeta_7^6$ & 95 & 2\\
    \bottomrule
    \end{tabular}
\end{table}

\begin{table}[H]
\footnotesize
    \centering
    \ContinuedFloat
    \caption[]{$W(E_8)$ (continued)}
    \begin{tabular}{c c c c c c}
    \toprule
    Type & Carter Symbol & Orbit Type & Eigenvalues & Bertini & $a(X)$ \\ 
    \midrule
    41 & $D_4\times A_1^2$ & $1^42^{34}6^{28}$ & 1, 1, $-1$, $-1$, $-1$, $-1$, $-\zeta_3$, $-\zeta_3^2$ & 10 & 0\\
    42 & $D_4 \times A_2$ & $2^63^86^{34}$ & 1, 1, $-1$, $-1$, $\zeta_3$, $\zeta_3^2$, $-\zeta_3$, $-\zeta_3^2$ &  42 & 0 \\
    43 & $D_4(a_1) \times A_2$  & $3^84^{18}12^{12}$ & 1, 1, $\zeta_3$, $\zeta_3^2$, $i$, $i$, $-i$, $-i$ & 93 & 2 \\
    44 & $D_5\times A_1$ & $1^22^{11}8^{27}$ & 1, 1, $-1$, $-1$, $\zeta_8$, $\zeta_8^3$, $-\zeta_8$, $-\zeta_8^3$ &  44 & 1\\
    45 & $D_5(a_1)\times A_1$ & $1^22^54^{15}6^812^6$ & 1, 1, $-1$, $-1$, $-\zeta_3$, $-\zeta_3^2$, $i$, $-i$ & 33 & 2\\
    46 & $D_6$ & $1^42^810^{22}$ & 1, 1, $-1$, $-1$, $-\zeta_5$, $-\zeta_5^2$, $-\zeta_5^3$, $-\zeta_5^4$ & 36 & 2\\
    47 & $D_6(a_1)$ & $1^44^58^{27}$ & 1, 1, $i$, $-i$, $\zeta_8$, $\zeta_8^3$, $-\zeta_8$, $-\zeta_8^3$ & 88 & 3\\
    48 & $D_6(a_2)$ & $1^42^46^{38}$ & 1, 1, $-1$, $-1$, $-\zeta_3$, $-\zeta_3$, $-\zeta_3^2$, $-\zeta_3^2$ & 30 & 3\\
    49 & $E_6$ & $1^63^612^{18}$ & 1, 1, $\zeta_3$, $\zeta_3^2$, $\zeta_{12}$, $\zeta_{12}^5$, $-\zeta_{12}$, $-\zeta_{12}^5$ & 102 & 2\\
    50 & $E_6(a_1)$ & $1^69^{26}$ & 1, 1, $\zeta_9$, $\zeta_9^2$, $\zeta_9^4$, $\zeta_9^5$, $\zeta_9^7$, $\zeta_9^8$ & 101 & 3\\
    51 & $E_6(a_2)$ & $1^63^66^{36}$ & 1, 1, $\zeta_3$, $\zeta_3^2$, $-\zeta_3$, $-\zeta_3$, $-\zeta_3^2$, $-\zeta_3^2$ & 87 & 4\\
    52 & $A_1^7$ & $1^22^{119}$ & 1, $-1$, $-1$, $-1$, $-1$, $-1$, $-1$, $-1$ & 2& $-5$\\
    53 & $A_2^3\times A_1$  & $2^33^{42}6^{18}$& 1, $-1$, $\zeta_3$, $\zeta_3$, $\zeta_3$, $\zeta_3^2$, $\zeta_3^2$, $\zeta_3^2$    & 82 & $-2$  \\
    54 & $A_3 \times A_1^4$  & $2^{30}4^{45}$ & 1, $-1$, $-1$, $-1$, $-1$, $-1$, $i$, $-i$  & 7 & $-3$\\
    55 & $A_3 \times A_2 \times A_1^2$    & $2^63^44^{15}6^612^{10}$ & 1, $-1$, $-1$, $-1$, $\zeta_3$, $\zeta_3^2$, $i$, $-i$ & 27  & $-2$\\
    56 & $A_3^2 \times A_1$ & $1^22^{11}4^{54}$ & 1, $-1$, $-1$, $-1$, $i$, $i$, $-i$, $-i$ & 25 & $-1$\\
    57 & $A_4 \times A_2 \times A_1$ & $2^13^25^66^210^415^630^2$ & 1, $-1$, $\zeta_3$, $\zeta_3^2$, $\zeta_5$, $\zeta_5^2$, $\zeta_5^3$, $\zeta_5^4$ & 81 & $-1$ \\
    58 & $A_4 \times A_3$ & $4^55^810^220^8$ & 1, $-1$, $i$, $-i$, $\zeta_5$, $\zeta_5^2$, $\zeta_5^3$, $\zeta_5^4$ & 76 & 0\\
    59 & $A_5 \times A_1^2$   &$2^63^26^{37}$ & 1, $-1$, $-1$, $-1$, $\zeta_3$, $\zeta_3^2$, $-\zeta_3$, $-\zeta_3^2$ & 23  & $-1$ \\
    60 & $A_5\times A_2$ & $1^22^23^86^{35}$ & 1, $-1$, $\zeta_3$, $\zeta_3$, $\zeta_3^2$, $\zeta_3^2$, $-\zeta_3$,  $-\zeta_3^2$ & 74 & 0\\
    61 & $A_6 \times A_1$ & $2^17^{18}14^8$ & 1, $-1$, $\zeta_7$, $\zeta_7^2$, $\zeta_7^3$, $\zeta_7^4$, $\zeta_7^5$, $\zeta_7^6$ & 79 & 0  \\
    62 & $(A_7)'$ & $1^22^14^58^{27}$ & 1, $-1$, $i$, $-i$, $\zeta_8$, $\zeta_8^3$, $-\zeta_8$, $-\zeta_8^3$ & 62 & 1\\
    63 & $(A_7)''$ &$4^28^{29}$ &  1, $-1$, $i$, $-i$, $\zeta_8$, $\zeta_8^3$, $-\zeta_8$, $-\zeta_8^3$ & 63 & 1 \\
    64 & $D_4\times A_1^3$ & $1^22^{35}6^{28}$ & 1, $-1$, $-1$, $-1$, $-1$, $-1$, $-\zeta_3$, $-\zeta_3^2$ & 6 & $-2$\\
    65 & $D_4 \times A_3$    & $2^64^{15}6^812^{10}$ & 1, $-1$, $-1$, $-1$, $-\zeta_3$, $-\zeta_3^2$, $i$, $-i$ & 21 & 0\\
    66 & $D_4(a_1)\times A_3$  & $2^24^{59}$ & 1, $-1$, $i$, $i$, $i$, $-i$, $-i$, $-i$ &  66 & 1  \\
    67 & $D_5 \times A_1^2$ &$2^{12}8^{27}$ & 1, $-1$, $-1$, $-1$, $\zeta_8$, $\zeta_8^3$, $-\zeta_8$, $-\zeta_8^3$ & 26 & $-1$ \\
    68 & $D_5 \times A_2$ &$3^46^28^924^6$ & 1, $-1$, $\zeta_3$, $\zeta_3^2$, $\zeta_8$, $\zeta_8^3$, $-\zeta_8$, $-\zeta_8^3$ & 77 & 0 \\
    69 & $D_5(a_1) \times A_2$  & $3^44^36^812^{14}$ & 1, $-1$, $\zeta_3$, $\zeta_3^2$, $-\zeta_3$, $-\zeta_3^2$, $i$, $-i$ &  69 & 1 \\
    70 & $D_6\times A_1$ & $1^22^910^{22}$ & 1, $-1$, $-1$, $-1$, $-\zeta_5$, $-\zeta_5^2$, $-\zeta_5^3$, $-\zeta_5^4$ & 22 & 0\\
    71 & $D_6(a_2)\times A_1$ & $1^22^56^{38}$ & 1, $-1$, $-1$, $-1$, $-\zeta_3$, $-\zeta_3$, $-\zeta_3^2$, $-\zeta_3^2$ & 18 & 1\\
    72 & $E_6\times A_1$ & $2^33^26^212^{18}$ & 1, $-1$, $\zeta_3$, $\zeta_3^2$, $\zeta_{12}$, $\zeta_{12}^5$, $-\zeta_{12}$, $-\zeta_{12}^5$  & 80  & 0\\
    73 & $E_6(a_1)\times A_1$  &$2^39^{14}18^6$ & 1, $-1$, $\zeta_9$, $\zeta_9^2$, $\zeta_9^4$, $\zeta_9^5$, $\zeta_9^7$, $\zeta_9^8$ & 78 & 1\\
    74 & $E_6(a_2) \times A_1$ & $2^33^26^{38}$& 1, $-1$, $\zeta_3$, $\zeta_3^2$, $-\zeta_3$, $-\zeta_3$, $-\zeta_3^2$, $-\zeta_3^2$ & 60 & 2 \\
    75 & $D_7$  & $2^24^212^{19}$& 1, $-1$, $i$, $-i$, $\zeta_{12}$, $\zeta_{12}^5$, $-\zeta_{12}$, $-\zeta_{12}^5$ &  75  & 1\\
    76 & $D_7(a_1)$ &$4^510^620^8$ & 1, $-1$, $i$, $-i$, $-\zeta_5$, $-\zeta_5^2$, $-\zeta_5^3$, $-\zeta_5^4$ & 58 & 2\\
    77 & $D_7(a_2)$ &$6^48^924^6$ & 1, $-1$, $-\zeta_3$, $-\zeta_3^2$, $\zeta_8$, $\zeta_8^3$, $-\zeta_8$, $-\zeta_8^3$ & 68 & 2 \\
    78 & $E_7$ & $1^22^218^{13}$ & 1, $-1$, $-\zeta_9$, $-\zeta_9^2$, $-\zeta_9^4$, $-\zeta_9^5$, $-\zeta_9^7$, $-\zeta_9^8$ & 73 & 1\\
    79 & $E_7(a_1)$ & $1^214^{17}$ & 1, $-1$, $-\zeta_7$, $-\zeta_7^2$, $-\zeta_7^3$, $-\zeta_7^4$, $-\zeta_7^5$, $-\zeta_7^6$ & 61 & 2\\
    80 & $E_7(a_2)$ & $1^22^26^312^{18}$ & 1, $-1$, $-\zeta_3$, $-\zeta_3^2$, $\zeta_{12}$, $\zeta_{12}^5$, $-\zeta_{12}$, $-\zeta_{12}^5$ & 72 & 2\\
    \bottomrule
    \end{tabular}
\end{table}

\newpage

\begin{table}[H]
\footnotesize
    \centering
    \ContinuedFloat
    \caption[]{$W(E_8)$ (continued)}
    \begin{tabular}{c c c c c c}
    \toprule
    Type & Carter Symbol & Orbit Type & Eigenvalues & Bertini & $a(X)$ \\ 
    \midrule
    81 & $E_7(a_3)$ & $1^26^310^730^5$ & 1, $-1$, $-\zeta_3$, $-\zeta_3^2$, $-\zeta_5$, $-\zeta_5^2$, $-\zeta_5^3$, $-\zeta_5^4$ & 57 & 3\\
    82 & $E_7(a_4)$ & $1^22^26^{39}$ & 1, $-1$, $-\zeta_3$, $-\zeta_3$, $-\zeta_3$, $-\zeta_3^2$, $-\zeta_3^2$, $-\zeta_3^2$ & 53 & 4\\
    83 & $A_1^8$ & $2^{120}$ & $-1$, $-1$, $-1$, $-1$, $-1$, $-1$, $-1$, $-1$ & 1 & $-7$\\
    84 & $A_2^4$ & $3^{80}$ & $\zeta_3$, $\zeta_3$, $\zeta_3$, $\zeta_3$, $\zeta_3^2$, $\zeta_3^2$, $\zeta_3^2$, $\zeta_3^2$ & 112 & $-3$    \\
    85 & $A_3^2\times A_1^2$ &$2^{12}4^{54}$ & $-1$, $-1$, $-1$, $-1$, $i$, $i$, $-i$, $-i$ & 15 & $-3$ \\
    86 & $A_4^2$ & $5^{48}$ & $\zeta_5$, $\zeta_5$, $\zeta_5^2$, $\zeta_5^2$, $\zeta_5^3$, $\zeta_5^3$, $\zeta_5^4$, $\zeta_5^4$ & 110 & $-1$ \\
    87 & $A_5 \times A_2 \times A_1$ &$2^33^86^{35}$  & $-1$, $-1$, $\zeta_3$, $\zeta_3$, $\zeta_3^2$, $\zeta_3^2$, $-\zeta_3$, $-\zeta_3^2$ & 51 & $-2$\\
    88 & $A_7 \times A_1$ &$2^24^58^{27}$  & $-1$, $-1$, $i$, $-i$, $\zeta_8$, $\zeta_8^3$, $-\zeta_8$, $-\zeta_8^3$ & 47 & $-1$ \\
    89 & $A_8$ & $3^29^{26}$ & $\zeta_3$, $\zeta_3^2$, $\zeta_9$, $\zeta_9^2$, $\zeta_9^4$, $\zeta_9^5$, $\zeta_9^7$, $\zeta_9^8$ & 108 & 0\\
    90 & $D_4 \times A_1^4$ & $2^{36}6^{28}$ & $-1$, $-1$, $-1$, $-1$, $-1$, $-1$, $-\zeta_3$, $-\zeta_3^2$ & 4 & $-4$\\
    91 & $D_4^2$ & $2^66^{38}$ & $-1$, $-1$, $-1$, $-1$, $-\zeta_3$, $-\zeta_3$, $-\zeta_3^2$, $-\zeta_3^2$ & 11 & $-1$ \\
    92 & $D_4(a_1)^2$ & $4^{60}$ & $i$, $i$, $i$, $i$, $-i$, $-i$, $-i$, $-i$  &  92 & 1 \\
    93 & $D_5(a_1)\times A_3$ & $4^{18}6^412^{12}$ & $-1$, $-1$, $-\zeta_3$, $-\zeta_3^2$, $i$, $i$, $-i$, $-i$ & 43 & 0 \\
    94 & $D_6 \times A_1^2$ & $2^{10}10^{22}$ & $-1$, $-1$, $-1$, $-1$, $-\zeta_5$, $-\zeta_5^2$, $-\zeta_5^3$, $-\zeta_5^4$ & 13 & $-2$ \\
    95 & $D_8$ & $2^114^{17}$ & $-1$, $-1$, $-\zeta_7$, $-\zeta_7^2$, $-\zeta_7^3$, $-\zeta_7^4$, $-\zeta_7^5$, $-\zeta_7^6$ & 40 & 0 \\
    96 & $D_8(a_1)$ & $4^312^{19}$ & $i$, $i$, $-i$, $-i$, $\zeta_{12}$, $\zeta_{12}^5$, $-\zeta_{12}$, $-\zeta_{12}^5$ &  96 & 1\\
    97 & $D_8(a_2)$ & $3^16^310^730^5$ & $-1$, $-1$, $-\zeta_3$, $-\zeta_3^2$, $-\zeta_5$, $-\zeta_5^2$, $-\zeta_5^3$, $-\zeta_5^4$ & 37 & 1  \\
    98 & $D_8(a_3)$  & $8^{30}$ & $\zeta_8$, $\zeta_8$, $\zeta_8^3$, $\zeta_8^3$, $-\zeta_8$, $-\zeta_8$, $-\zeta_8^3$, $-\zeta_8^3$ &  98 & 1 \\
    99 & $E_6\times A_2$ & $3^812^{18}$ & $\zeta_3$, $\zeta_3$, $\zeta_3^2$, $\zeta_3^2$, $\zeta_{12}$, $\zeta_{12}^5$, $-\zeta_{12}$, $-\zeta_{12}^5$ & 111 & $-1$ \\
    100 & $E_6(a_2)\times A_2$ & $3^86^{36}$ &  $\zeta_3$, $\zeta_3$, $\zeta_3^2$, $\zeta_3^2$, $-\zeta_3$, $-\zeta_3$, $-\zeta_3^2$, $-\zeta_3^2$ &  100 & 1 \\
    101 & $E_7 \times A_1$ & $2^318^{13}$ & $-1$, $-1$, $-\zeta_9$, $-\zeta_9^2$, $-\zeta_9^4$, $-\zeta_9^5$, $-\zeta_9^7$, $-\zeta_9^8$ & 50 & $-1$\\
    102 & $E_7(a_2)\times A_1$ & $2^36^312^{18}$ & $-1$, $-1$, $-\zeta_3$, $-\zeta_3^2$, $\zeta_{12}$, $\zeta_{12}^5$, $-\zeta_{12}$, $-\zeta_{12}^5$ & 49 & 0 \\ 
    103 & $E_7(a_4)\times A_1$ & $2^36^{39}$ & $-1$, $-1$, $-\zeta_3$, $-\zeta_3$, $-\zeta_3$, $-\zeta_3^2$, $-\zeta_3^2$, $-\zeta_3^2$ & 31 & 2  \\
    104 & $E_8$  & $30^8$ & $-\zeta_{15}$, $-\zeta_{15}^2$, $-\zeta_{15}^4$, $-\zeta_{15}^7$, $-\zeta_{15}^8$, $-\zeta_{15}^{11}$, $-\zeta_{15}^{13}$, $-\zeta_{15}^{14}$  & 109 & 0 \\
    105 & $E_8(a_1)$  & $24^{10}$ & $\zeta_{24}$, $\zeta_{24}^5$, $\zeta_{24}^7$, $\zeta_{24}^{11}$, $-\zeta_{24}$, $-\zeta_{24}^5$, $-\zeta_{24}^7$, $-\zeta_{24}^{11}$ &  105 &  1\\
    106 & $E_8(a_2)$  & $20^{12}$ & $\zeta_{20}$, $\zeta_{20}^3$, $\zeta_{20}^7$, $\zeta_{20}^9$, $-\zeta_{20}$, $-\zeta_{20}^3$, $-\zeta_{20}^7$, $-\zeta_{20}^9$ & 106 & 1 \\
    107 & $E_8(a_3)$  &  $12^{20}$ & $\zeta_{12}$, $\zeta_{12}$, $\zeta_{12}^5$, $\zeta_{12}^5$, $-\zeta_{12}$, $-\zeta_{12}$, $-\zeta_{12}^5$, $-\zeta_{12}^5$ &  107 & 1 \\
    108 & $E_8(a_4)$  & $6^118^{13}$ & $-\zeta_3$, $-\zeta_3^2$, $-\zeta_9$, $-\zeta_9^2$, $-\zeta_9^4$, $-\zeta_9^5$, $-\zeta_9^7$, $-\zeta_9^8$ & 89 & 0\\
    109 & $E_8(a_5)$ &  $15^{16}$ & $\zeta_{15}$, $\zeta_{15}^2$, $\zeta_{15}^4$, $\zeta_{15}^7$, $\zeta_{15}^8$, $\zeta_{15}^{11}$, $\zeta_{15}^{13}$, $\zeta_{15}^{14}$ & 104 & 2 \\
    110 & $E_8(a_6)$ & $10^{24}$  & $-\zeta_5$, $-\zeta_5$, $-\zeta_5^2$, $-\zeta_5^2$, $-\zeta_5^3$, $-\zeta_5^3$, $-\zeta_5^4$, $-\zeta_5^4$ & 86 & 3\\
    111 & $E_8(a_7)$ & $6^412^{18}$  &  $-\zeta_3$, $-\zeta_3$, $-\zeta_3^2$, $-\zeta_3^2$, $\zeta_{12}$, $\zeta_{12}^5$, $-\zeta_{12}$, $-\zeta_{12}^5$  & 99 & 3 \\
    112 & $E_8(a_8)$ &  $6^{40}$ & $-\zeta_3$, $-\zeta_3$, $-\zeta_3$, $-\zeta_3$, $-\zeta_3^2$, $-\zeta_3^2$, $-\zeta_3^2$, $-\zeta_3^2$ & 84  & 5\\
    \bottomrule
    \end{tabular}
\end{table}

\begin{table}[H]
    \centering
    \caption{$H^1(G,\mathrm{Pic}(\overline{X}))$}
    \label{Table H1}
    \begin{tabular}{c c c }
    \toprule
    degree 1 & degree 2 &  $H^1(G,\mathrm{Pic}(\overline{X}))$ \\ 
    \midrule
    $(A_1^4)'$ & $(A_1^4)''$ & $\bZ/2\bZ \times \bZ/2\bZ$\\
    $(A_1^4)''$ & $(A_1^4)'$ & 0 \\
    $(A_3\times A_1^2)'$ & $(A_3 \times A_1^2)''$ & $\bZ/2\bZ \times \bZ/2\bZ$\\
    $(A_3\times A_1)''$ & $(A_3 \times A_1^2)'$ & 0\\
    $(A_3^2)'$ & $A_3^2$ & $\bZ/2\bZ \times \bZ/2\bZ$\\
    $(A_3^2)''$ & / & 0\\
    $(A_5\times A_1)'$ & $(A_5 \times A_1)''$ & $\bZ/2\bZ \times \bZ/2\bZ$\\
    $(A_5 \times A_1)''$ & $(A_5 \times A_1)'$ & 0\\
    $(A_7)'$ & $A_7$ & $\bZ/2\bZ \times \bZ/2\bZ$\\
    $(A_7)''$ & / & 0\\
    \bottomrule
    \end{tabular}
\end{table}

\begin{table}[H]
    \centering
    \captionsetup{labelsep=newline, justification=centering} 
    
    \begin{minipage}[t]{0.45\textwidth}
        \centering
        \caption{Non-minimal open cases}
        \label{Table nonmin}
        \vspace{1mm} 
        \begin{tabular}[t]{c  c}
        \toprule
        Type & Bertini Twist\\
        \midrule
        79 & 61 \\
        42 & 42 \\
        43 & 93 \\
        58 & 76* \\
        68 & 77* \\
        69 & 69 \\
        \bottomrule
        \end{tabular}
    \end{minipage}
    \hfill 
    \begin{minipage}[t]{0.45\textwidth}
        \centering
        \caption{Minimal open cases}
        \label{Table min}
        \vspace{1mm}
        \begin{tabular}[t]{c  c}
        \toprule
        Type & Bertini Twist\\
        \midrule
        66* & 66* \\
        75* & 75* \\
        84 & 112 \\
        86 & 110 \\
        89 & 108 \\
        92 & 92 \\
        96 & 96 \\
        98 & 98 \\
        99 & 111 \\
        100 & 100 \\
        104 & 109 \\
        105 & 105 \\
        106 & 106 \\
        107 & 107 \\
        \bottomrule
        \end{tabular}
    \end{minipage}
    
\end{table}

We now list the normal forms that are a consequence of the results in \S\! \ref{sec: sextics}.

\begin{lem}
    Let $X$ be a del Pezzo surface of degree $1$ over $\bF_3$. Denote by $f_i\in \bF_3[x,y]$ a homogeneous polynomial of degree $i$. Then $X$ is isomorphic to a surface in one of the following five families:
    \begin{enumerate}
        \item $V(w^2+z^3+z^2(x^2+y^2)+z(ax^3y+bx^2y^2)+f_6(x,y))$ for $a,b \in \bF_3,$
        \item $V(w^2+z^3+z^2(x^2+2y^2)+z(ax^3y+bx^2y^2)+f_6(x,y))$ for $a,b \in \bF_3,$
        \item $V(w^2+z^3+z^2x^2+z(axy^3+by^4)+f_6(x,y))$ for $a,b \in \bF_3,$
        \item $V(w^2+z^3+z^22x^2+z(axy^3+by^4)+f_6(x,y))$ for $a,b \in \bF_3,$
        \item $V(w^2+z^3+zf_4(x,y)+ax^5y+bx^4y^2+cx^2y^4+dxy^5)$ for $a,b,c,d \in \bF_3$.
    \end{enumerate}
\end{lem}

\begin{lem}
    Let $X$ be a del Pezzo surface of degree $1$ over $\bF_4$. Denote by $f_i \in \bF_4[x,y]$ a homogeneous polynomial of degree $i$. Then $X$ is isomorphic to a surface of one of the following $15$ families:
    \begin{enumerate}
        \item $V(w^2+z^3+wzx+wf_3(x,y)+azy^4+f_6(x,y))$ for $a \in \bF_4$,
        \item $V(w^2+z^3+zf_4(x,y)+ax^5y+bx^3y^3+cxy^5)$ for $a,b,c \in \bF_4$,
        \item $V(w^2+z^3+wx^3+zf_4(x,y)+ax^6+bx^4y^2+cx^3y^3+dxy^5)$ for $a,b,c,d \in \bF_4$,
        \item $V(w^2+z^3+w \alpha x^3+zf_4(x,y)+ax^6+bx^4y^2+cx^3y^3+dxy^5)$ for $a,b,c,d \in \bF_4$,
        \item $V(w^2+z^3+w\alpha^2x^3+zf_4(x,y)+ax^6+bx^4y^2+cx^3y^3+dxy^5)$ for $a,b,c,d \in \bF_4$,
        \item $V(w^2+z^3+wxy^2+zf_4(x,y)+ax^5y+bx^4y^2+cxy^5+dy^6)$, for $a,b,c,d \in \bF_4$,
        \item $V(w^2+z^3+w(x^3+y^3)+zf_4(x,y)+ax^6+bx^2y^4+cxy^5+dy^6)$, for $a,b,c,d \in \bF_4$,
        \item $V(w^2+z^3+w\alpha(x^3+y^3)+zf_4(x,y)+ax^6+bx^2y^4+cxy^5+dy^6)$, for $a,b,c,d \in \bF_4$,
        \item $V(w^2+z^3+w\alpha^2(x^3+y^3)+zf_4(x,y)+ax^6+bx^2y^4+cxy^5+dy^6)$, for $a,b,c,d \in \bF_4$,
        \item $V(w^2+z^3+w(x^3+x^2y+\alpha xy^2)+zf_4(x,y)+ax^6+bx^2y^4+cxy^5+dy^6)$, for $a,b,c,d \in \bF_4$,
        \item $V(w^2+z^3+w\alpha(x^3+x^2y+\alpha xy^2)+zf_4(x,y)+ax^6+bx^2y^4+cxy^5+dy^6)$, for $a,b,c,d \in \bF_4$,
        \item $V(w^2+z^3+w\alpha^2(x^3+x^2y+\alpha xy^2)+zf_4(x,y)+ax^6+bx^2y^4+cxy^5+dy^6)$, for $a,b,c,d \in \bF_4$,
        \item $V(w^2+z^3+w(x^3+\alpha y^3)+zf_4(x,y)+ax^6+bx^2y^4+cxy^5+dy^6)$, for $a,b,c,d \in \bF_4$,
        \item $V(w^2+z^3+w\alpha(x^3+\alpha y^3)+zf_4(x,y)+ax^6+bx^2y^4+cxy^5+dy^6)$, for $a,b,c,d \in \bF_4$,
        \item $V(w^2+z^3+w\alpha^2(x^3+\alpha y^3)+zf_4(x,y)+ax^6+bx^2y^4+cxy^5+dy^6)$, for $a,b,c,d \in \bF_4$.
    \end{enumerate}
\end{lem}

\end{document}